\documentclass[11pt]{article}

\usepackage[letterpaper, hmargin=1in, top=1in, bottom=1in, footskip=0.6in]{geometry}

\usepackage{titlesec}
\titleformat*{\section}{\Large\bfseries}
\titleformat{\subsection}[runin]{\normalfont\bfseries}{\thesubsection.}{.5em}{}[.]\titlespacing{\subsection}{0pt}{2ex plus .1ex minus .2ex}{.8em}
\titleformat{\subsubsection}[runin]{\normalfont\bf}{\thesubsubsection.}{.3em}{}[.]\titlespacing{\subsubsection}{0pt}{1ex plus .1ex minus .2ex}{.5em}
\titleformat{\paragraph}[runin]{\normalfont\itshape}{\theparagraph.}{.3em}{}[.]\titlespacing{\paragraph}{0pt}{1ex plus .1ex minus .2ex}{.5em}

\usepackage{amsmath}
\usepackage{amssymb}
\usepackage{amsfonts}
\usepackage{latexsym}
\usepackage{amsthm}
\usepackage{amsxtra}
\usepackage{amscd}
\usepackage{bbm}
\usepackage{mathrsfs}
\usepackage{bm}
\usepackage{tensor}

\usepackage{graphicx, color}

\definecolor{darkred}{rgb}{0.9,0,0.3}
\definecolor{darkblue}{rgb}{0,0.3,0.9}

\definecolor{vdarkred}{rgb}{0.7,0,0.2}
\definecolor{vdarkblue}{rgb}{0,0.2,0.7}
\usepackage[pdftex, colorlinks, linkcolor=vdarkblue,urlcolor=vdarkblue,citecolor=vdarkred]{hyperref}

\usepackage[nottoc,notlof,notlot]{tocbibind}
\usepackage{cite}

\numberwithin{equation}{section}
\numberwithin{figure}{section}

\theoremstyle{plain} 
\newtheorem{theorem}{Theorem}[section]
\newtheorem*{theorem*}{Theorem}
\newtheorem{lemma}[theorem]{Lemma}
\newtheorem*{lemma*}{Lemma}
\newtheorem{corollary}[theorem]{Corollary}
\newtheorem*{corollary*}{Corollary}
\newtheorem{proposition}[theorem]{Proposition}
\newtheorem*{proposition*}{Proposition}
\newtheorem{conjecture}[theorem]{Conjecture}
\newtheorem*{conjecture*}{Conjecture}

\theoremstyle{definition} 

\newtheorem*{definition*}{Definition}

\newtheorem*{example*}{Example}
\newtheorem{remark}[theorem]{Remark}
\newtheorem*{remark*}{Remark}

\newtheorem*{assumption*}{Assumption}

\renewcommand{\b}[1]{\boldsymbol{\mathrm{#1}}} 
\newcommand{\bb}{\mathbb} 
\renewcommand{\cal}{\mathcal}

\newcommand{\ul}[1]{\underline{#1} \!\,} 

\newcommand{\e}{\mathrm{e}}
\newcommand{\A}{\mathcal{A}}
\newcommand{\D}{\mathrm{D}}

\newcommand{\ii}{\mathrm{i}}
\newcommand{\dd}{\mathrm{d}}

\newcommand*{\deq}{\mathrel{\vcenter{\baselineskip0.65ex \lineskiplimit0pt \hbox{.}\hbox{.}}}=}
\newcommand*{\eqd}{=\mathrel{\vcenter{\baselineskip0.65ex \lineskiplimit0pt \hbox{.}\hbox{.}}}}

\renewcommand{\leq}{\leqslant}

\renewcommand{\geq}{\geqslant}

\renewcommand{\epsilon}{\varepsilon}

\newcommand{\ceil}[1]  {\lceil  {#1} \rceil}

\DeclareMathOperator{\tr}{Tr}

\DeclareMathOperator{\supp}{supp}

\DeclareMathOperator{\im}{Im}

\newcommand*{\rom}[1]{\expandafter\@slowromancap\romannumeral #1@}

\begin{document}	
\begin{center}
	\vspace{5em}
	{\Large\bf
	Spectral gap and edge universality of dense random regular graphs}\\ \vspace{3em}
\large Yukun He \\ \vspace{1em}
\end{center}

\begin{abstract}
\noindent	Let $\mathcal A$ be the adjacency matrix of a random $d$-regular graph on $N$ vertices, and we denote its eigenvalues by $\lambda_1\geq \lambda_2\cdots \geq \lambda_{N}$. For $N^{2/3+o(1)}\leq d\leq N/2$, we prove optimal rigidity estimates of the extreme eigenvalues of $\mathcal A$, which in particular imply that
\[
\max\{|\lambda_N|,\lambda_2\} <2\sqrt{d-1}
\]
with very high probability. In the same regime of $d$, we also show that
\[
N^{2/3}\bigg(\frac{\lambda_2+d/N}{\sqrt{d(N-d)/N}}-2\bigg) \overset{d}{\longrightarrow} \mathrm{TW}_1\,,
\]  
where $\mathrm{TW}_1$ is the Tracy-Widom distribution for GOE; analogue results also hold for other non-trivial extreme eigenvalues.
\end{abstract}

\hfill{\textit{To the memory of my waigong, Sun Wenya}}

\section{Introduction}

In this article, we consider a random $d$-regular graph on $N$ vertices, under the uniform probability measure. Let $\cal A \in \bb R^{N\times N}$ be the adjacency matrix of the graph, and we denote its eigenvalues by $\lambda_1\geq \cdots \geq \lambda_N$. It is easy to see that 
$
\lambda_1=d\,
$
with corresponding eigenvector $\b e\deq N^{-1/2}(1,1,...,1)^*$.

The behavior of nontrivial extreme eigenvalues of $\A$ is of particular interest in graph theory and computer science. For instance, the gap between the first and second eigenvalues measures the expanding property of the graph. For a deterministic $d$-regular graph on $N$ vertices, the Alon-Boppana bound \cite{Alo86} states that 
\[
\lambda_2, |\lambda_N| \geq 2\sqrt{d-1}(1-o(1))
\]
for $d$ fixed and $N$ large enough. A \textit{Ramanujan graph} is a $d$-regular graph whose nontrivial eigenvalues are bounded in absolute value by $2\sqrt{d-1}\,,$ i.e.\,it is a graph that essentially saturates the Alon–Boppana bound. Ramanujan graphs were first constructed by Lubotzky, Phillips and Sarnak \cite{LPS88}, and by Margulis \cite{Mar88} for some values of $d$. The construction of Ramanujan graphs in the bipartite case for all degrees was given by Marcus, Spielman and Srivastava \cite{MSS13,MSS13.2}. For the random $d$-regular graph $\A$, when $d$ is fixed, Friedman \cite{Fri08} showed that, for sufficiently large $N$,  
\[
\lambda_2, |\lambda_N| \leq 2\sqrt{d-1}(1+o(1))
\] 
with high probability (the proof was later substantially simplified by Bordenave \cite{Bor15}).  This means that a random $d$-regular graph is typically ``almost Ramanujan". More recently, Huang, McKenzie and Yau \cite{HMY24} (following Huang and Yau \cite{HY21})  extended this result by showing the near-optimal rate
\[
\lambda_2, |\lambda_N| \leq 2\sqrt{d-1}(1+O(N^{-2/3+o(1)}))
\]
with probability $1-N^{-1+o(1)}$. 

The case of $d \leq  N/2$ that tends to infinity with $N$ was conjectured by Vu \cite{Vu14} to have
\begin{equation} \label{vuconjecture}
\lambda_2, |\lambda_N| =2\sqrt{d(N-d)/N})(1+o(1))
\end{equation}
with high probability. The magnitude bound $\lambda_2+|\lambda_N|=O(\sqrt{d})$ with high probability was proved by Broder, Frieze, Suen and Upfal \cite{BFSU98}  for $d=o(\sqrt{N})$; by Cook, Goldstein and Johnson \cite{CGJ15} for $d=O(N^{2/3})$; by Tikhomirov and Youssef \cite{TY19} for all $d \leq N/2$. The eigenvalue locations were proved to satisfy $\lambda_2, |\lambda_N| =2\sqrt{d-1}(1+o(1))$ in the regime $N^{o(1)}\leq  d \leq N^{2/3-o(1)}$, by Bauerschmidt, Huang, Knowles and Yau \cite{BHKY19}. Very recently, Sarid \cite{Sarid22} proved \eqref{vuconjecture} for $1\ll d\leq cN$, where $c$ is a small constant. 

Our first main result determines the extreme eigenvalue locations in the regime $N^{2/3+o(1)}\leq d \leq N/2$, with optimal error bounds. Together with \cite{BHKY19,Sarid22}, we settle the conjecture \eqref{vuconjecture} in the whole regime $1 \ll d \leq N/2$. We may now state our first main result.
\begin{theorem} \label{thm rigidity}
Fix $\tau>0$, $k\geq 2$, and assume $N^{2/3+\tau}\leq d\leq N/2$. For any fixed $\varepsilon,D>0$, we have
\begin{equation}\label{2.11}
	\lambda_2,...,\lambda_k=(2\sqrt{d(N-d)/N}-d/N)(1+O(N^{-2/3+\varepsilon}))
\end{equation}
as well as
\[
\lambda_N,...,\lambda_{N-k}=(-2\sqrt{d(N-d)/N}-d/N)(1+O(N^{-2/3+\varepsilon}))
\]
with probability $1-O(N^{-D})$.
\end{theorem}

The negative shift $-d/N$ in Theorem \ref{thm rigidity} is only relevant if we want the optimal error bound $O(N^{-2/3+\varepsilon})$. As $\sqrt{d(N-d)/N}\asymp d^{1/2}$ for $d\leq N/2$, Theorem \ref{thm rigidity} implies 
\begin{equation*} 
	\lambda_2, |\lambda_N| =2\sqrt{d(N-d)/N})(1+O(N^{-1/2}))
\end{equation*}
 with very high probability. In addition, Theorem \ref{thm rigidity} implies that for $N^{2/3+o(1)} \leq d \leq N/2$, almost all $d$-regular graphs on $N$ vertices are Ramanujan. Indeed, by \eqref{2.11} we have 
\[
\lambda_2-2\sqrt{d-1}=\frac{2-2d^2/N}{\sqrt{d(N-d)/N}+\sqrt{d-1}}-\frac{d}{N}+O(d^{1/2}N^{-2/3+\varepsilon})
\]
with very high probability. As $N^{2/3+o(1)} \leq d \leq N/2$, the above is negative with very high probability. The analogue also holds for $-\lambda_N-2\sqrt{d-1}$. This yields the following result.

\begin{corollary} \label{cor1.2}
Fix $D,\tau>0$. For $d$ large enough and $2d \leq N \leq d^{3/2+\tau}$, 
\begin{equation} \label{1.2}
	\mathbb P\Big(\max\{|\lambda_N|,\lambda_2\} <2\sqrt{d-1}\Big) \geq 1-N^{-D}\,.
\end{equation}
\end{corollary}

Beyond the law of large numbers, the distributions of the extreme eigenvalues of $\A$ were conjectured in \cite{MNS08} to satisfy \textit{edge universality}, i.e.\,after normalization, their joint distribution is the same as that of the extreme eigenvalues of the Gaussian Orthogonal Ensemble. Edge universality was proved by by Bauerschmidt, Huang, Knowles and Yau \cite{BHKY19} for $\A$ in the intermediate regime $N^{2/9+o(1)} \leq d \ll N^{1/3-o(1)}$. The authors showed that
\begin{equation} \label{1.5}
	N^{2/3}\bigg(\frac{\lambda_2}{\sqrt{d-1}}-2\bigg) \overset{d}{\longrightarrow} \mathrm{TW}_1\,,
\end{equation}
together with analogue results for other extreme eigenvalues. Recently, Huang and Yau \cite{HY23} extended \eqref{1.5} to $N^{o(1)}\leq d\leq N^{1/3-o(1)}$. Our second main result is the edge universality of $\A$ in the dense regime $N^{2/3+o(1)}\leq d \leq N/2$.

\begin{theorem} \label{theorem main result}
Fix $\tau>0$ and assume $N^{2/3+\tau}\leq d\leq N/2$. Let $\mu_1\geq \cdots \geq \mu_N$ denote the eigenvalues of a Gaussian Orthogonal Ensemble. Fix $k \in \bb N_+$.  We have
	\begin{multline*}
	\lim_{N\to \infty}	\bb P_{\A}\bigg(N^{2/3}\bigg(\frac{\lambda_{i+1}+d/N}{\sqrt{d(N-d)/N}}-2\bigg) \geq s_i, N^{2/3}\bigg(\frac{\lambda_{N-i+1}+d/N}{\sqrt{d(N-d)/N}}+2\bigg)\geq r_i, 1\leq i \leq k\bigg)\\
	=\lim_{N\to \infty}	\bb P_{\emph{GOE}}\Big(N^{2/3}(\mu_{i}-2)\geq s_i, N^{2/3}(\mu_{N-i+1}+2)\geq r_i, 1\leq i \leq k\Big)
	\end{multline*}
uniformly for all $s_1,r_1,...,s_k,r_k \in \bb R$.
\end{theorem}
To prove the main results, we analysis the Stieltjes transform of $\A$ near the spectral edge, on all mesoscopic spectral scales. This \textit{Green function method} is widely used in the random matrix community. To start of, it was applied to Wigner matrices, in particular in \cite{TV1,TV2,ESY1,ESY2,BEYY14,EPRSY,EYY3,EYY1}. It was then applied in \cite{EKYY1,EKYY2,LS1,HLY,HK20,BHY,Lee21,HLY15, H19} to sparse matrices, which includes the adjacency matrix of sparse Erd\H{o}s-R\'{e}nyi graphs $\cal G(N,p)$ for $p \gg N^{-1}$. These works rely on the fact that the matrix entries are independent (subject to the symmetry constraint), which is not the case for $\A$. In the work \cite{BKY15}, the authors developed a technique through local switching, which opens the door of studying random regular graphs through the Green function method. For $N^{o(1)} \leq  d \leq N^{2/3-o(1)}$, they proved that the eigenvalues of $\A$ satisfy the local semicircle law. The idea of switching was then applied to prove various results for $\A$ in the regime $M^{o(1)}\leq d \leq N^{2/3-o(1)}$ \cite{BHKY15,BHKY19}, and $d$ fixed \cite{HY21,BHY19}. All these works require the degree upper bound $d \ll N^{2/3}$, which is essentially due to the approximation $1-\A_{ij} \approx 1$. In other words, due to the sparsity of the graph in the regime $d \ll N^{2/3}$, in many situations, one can take two vertices of the graph, and with an affordable error assume that they are disconnected.

In order to deal with the dense case $N^{2/3+o(1)}\leq d \leq N/2$, we develop an algorithm which is insensitive to the increasing density of the graph. Comparing to \cite{BKY15,BHKY19}, the integration by parts formula used in this paper (see Lemma \ref{lem2.2}) comes with an error term that does not explicitly depend on $d$. Another ingredient of the proof is a large deviation result on the powers of $\A$ (see Proposition \ref{prop4.4}), which essentially counts the number of short cycles of the graph. This enables us to replace the entries of $\A^r$ ($r\geq 2$) by their expectations, with affordable errors. 

Our first step is to prove a weak local semicircle law for all $N^{o(1)} \leq d \leq N/2$, which is stated in terms of Green functions (see Theorem \ref{theorem 4.1}). A standard consequence of Theorem \ref{theorem 4.1} is the following complete eigenvector delocalization.

\begin{corollary} \label{cor4.2}
Fix $\tau>0$ and assume $N^{\tau}\leq d \leq N/2$. Let $\b u_i\in \mathbb S^{N-1}$ denote the $i$-th eigenvector of $\A$. For any fixed $\varepsilon,D>0$, we have
	\[
	\max_i	\|\b u_i\|_\infty =O(N^{-1/2+\varepsilon})
	\]
	with probability $1-O(N^{-D})$.
\end{corollary}

  After obtaining the weak local law, we perform a refined analysis of the averaged self-consistent equations near the spectral edge (see Proposition \ref{prop4.5}). This leads to a strong estimate on the traces of the Green functions in the regime $N^{2/3+o(1)}\leq d \leq N/2$ (see Proposition \ref{propsition 6.1}), and from which Theorem \ref{thm rigidity} follows. Providing optimal edge rigidity, the edge universality Theorem \ref{theorem main result} is proved basing on the usual three-step approach of random matrix theory \cite{EY17}. The same strategy was also used in \cite{BHKY19}.

	From Theorems \ref{thm rigidity} and \ref{theorem main result}, we see a shift of $-d/N$ on both spectral edges of $\cal A$. This is due to the fact that the diagonal entries of the adjacency matrix are 0. More precisely, observe that
	\[
	N\b e \b e^*-\cal A-I
	\]
	is the adjacency matrix of a random $(N-1-d)$-regular graph on $N$ vertices, and we denote its eigenvalues by $N-1-d=\widehat{\lambda}_1\geq \cdots \geq \widehat{\lambda}_N$. Thus for $ 2\leq k \leq N$, we have the relation
	\[
	\lambda_k+\widehat{\lambda}_{N+2-k}=-1\,.
	\]
	Our main results suggest that the shift $-1$ is shared between $\lambda$ and $\widehat{\lambda}$, with the amount proportional to the graph degree. The shift is essential to getting the Tracy-Widom limit, as 
	\[
	\frac{1}{\sqrt{d(N-d)/N}}\cdot \frac{d}{N} \gg N^{-2/3}
	\]
	for $ N^{2/3}\ll d \leq N/2$.
	
	Comparing the parameters of \eqref{1.5} and Theorem \ref{theorem main result}, together with the degree symmetry $d \longleftrightarrow N-d-1$ for $d-$regular graphs, one could propose that edge universality holds for all non-trivial random $d$-regular graphs, in the following scaling.

\begin{conjecture} \label{con1.5}
	Assume $3 \leq d \leq N-4$. There exists a constant $c_{N,d}$\footnote{Presumably, this constant is non-trivial only when $\min(d,N-d)$ is bounded. As observed in \cite{BHY19,HY23} and the current paper, $c_{N,d}=0$ if $\min(d,N-d)\geq N^{o(1)}$.} such that
	\[
	N^{2/3}\bigg(\frac{\lambda_2+d/N}{\sqrt{(d-1)(N-d-2)/N}}-2\bigg)-c_{N,d} \overset{d}{\longrightarrow} \mathrm{TW}_1\,,
	\]  
	where $\mathrm{TW}_1$ is the Tracy-Widom distribution for GOE. Analogues results also hold for other non-trivial extreme eigenvalues.
\end{conjecture}

Although the proof of Conjecture \ref{con1.5} in the regime when $d$ is fixed, is apparently difficult, it is probable that combining the techniques of \cite{BHKY19,HY23} and the current paper, one can prove optimal edge rigidity and universality for $N^{o(1)}\leq d \leq N/2$. Providing this is the case, the following will also stand.

\begin{conjecture}
For $d$ large enough and $N \geq 2d$, \eqref{1.2} holds if and only if $N \ll d^3$.
\end{conjecture}

The rest of this paper is organized as follows. In section \ref{sec2} we recall the local switching, and prove an integration by parts formula which is insensitive to the degree $d$. In Section \ref{sec3} we prove a large deviation result on the powers of $\A$. In Section \ref{sec4} we prove the weak local semicircle law for all $ N^{o(1)}\leq d \leq N/2$. In Section \ref{sec5} we prove a strong self-consistent equation near the spectral edge. Finally in Section \ref{sec6} we use the results in Sections \ref{sec4} and \ref{sec5} to conclude the proof of our main results.

\subsection*{Conventions}. Unless stated otherwise, all quantities depend on the fundamental large parameter $N$, and we omit this dependence from our notation. We use the usual big $O$ notation $O(\cdot)$, and if the implicit constant depends on a parameter $\alpha$ we indicate it by writing $O_\alpha(\cdot)$. Let 
\[
X=(X^{(N)}(u): N \in \bb N, u \in U^{(N)})\,, \quad Y=(Y^{(N)}(u): N \in \bb N, u \in U^{(N)})
\]
be two families of nonnegative random variables, where $U^{(N)}$ is a possibly $N$-dependent parameter
set, and $Y\geq 0$. We say that $X$ is stochastically dominated by $Y$ , uniformly in $u$, if for any fixed $\varepsilon,D>0$,
\[
\sup_{u \in U^{(N)}}\bb P(|X|\geq Y N^{\varepsilon}) =O_{\varepsilon,D} (N^{-D})\,.
\]
We write $X\asymp Y$ if $X =O(Y)$ and $Y=O(X)$. If $X$ is stochastically dominated by $Y$ , we use the notation $X \prec Y$, or equivalently $X=O_{\prec}(Y)$. We say an event $\Omega$ holds with very high probability if for any $D>0$, $1-\bb P(\Omega)=O_D(N^{-D})$.

\subsection*{Acknowledgment}The author would like to thank Zhigang Bao for helpful comments. The author is supported by NSFC No.\,2023YFA1010400,  NSFC No.\,12322121, Hong Kong RGC Grant No.\,21300223 and CityU Start-up Grant No.\,7200727.

\section{Local switchings} \label{sec2}
As in \cite{BHKY19,BKY15}, we rely on switchings for regular graphs and the invariance under the permutation of vertices. For indices $i,j,k,l$, we define the signed adjacency matrices
\begin{equation} \label{2.3}
	(\Delta_{ij})_{xy}\deq \delta_{ix}\delta_{jy}+\delta_{iy}\delta_{jx}\,, \quad \xi_{ij}^{kl}=\Delta_{ij}+\Delta_{kl}-\Delta_{ik}-\Delta_{jl}\,.
\end{equation}
In addition, we denote the indicator function that the edges $ij$ and $kl$ are switchable by
\begin{equation} \label{2.4}
	\chi_{ij}^{kl}(\cal A)=\cal A_{ij}(1-\cal A_{ik})\cal A_{kl}(1-\cal A_{jl})\,.
\end{equation}
The following identity is a consequence of the uniform probability measure on $\cal A$. The proof is given in \cite[Proposition 3.1]{BHKY19}.

\begin{lemma} \label{lem2.1}
	Let $i,j,k,l$ be distinct indices. Let $F$ be a function which depends on the random graph $\cal A$, and possibly on the indices $i,j,k,l$. We have
	\[
	\bb E F(\cal A)\chi_{ij}^{kl}(\cal A)=\bb E F(\cal A+\xi_{ij}^{kl})\chi_{ik}^{jl}(\cal A)\,,
	\]
	where $\xi$ and $\chi$ are defined in \eqref{2.3} and \eqref{2.4} respectively.
\end{lemma} 

Let us abbreviate
\begin{equation}\label{M}
	\cal M_{ij}(F(\A))\deq \max_{kl} |F(\A+\xi_{ij}^{kl})|\,.
\end{equation}
The next result improves \cite[Corollary 3.2]{BHKY19} to adapt the dense graph setting. This is the main formula we use in generating non-trivial self-consistent equations.

\begin{lemma} \label{lem2.2}
Let $i,j$ be distinct indices. Let $F$ be a function which depends on the random graph $\cal A$, and possibly on the indices $i,j$. We have
	\begin{equation*}
		\begin{aligned}
			\bb E \cal A_{ij} F(\cal A)=&\,\frac{1}{(N-d)d}\sum_{kl}\bb E \chi_{ik}^{jl}(\cal A)\big(F(\cal A+\xi_{ij}^{kl})-F(\A)\big)	+\frac{d}{N-d}\bb E F(\A)\\
			&-\frac{1}{(N-d)d}\bb E (\A^3)_{ij}F(\A)+O(N^{-1})\cdot\bb E\cal M_{ij}(F(\A))\,.
	\end{aligned}
\end{equation*}
We often refer the last term above as the \emph{remainder term}.
\end{lemma}

\begin{proof}
	Since $\sum_{k} \cal A_{ik}=\sum_{l}\cal A_{kl}=d$, we have
	\begin{equation} \label{2.6}
		\begin{aligned}
		\bb E \cal A_{ij}F(\cal A )&=\frac{1}{(N-d)d}\sum_{kl}\bb E \cal A_{ij}(1-\cal A_{ik})\cal A_{kl}F(\cal A)\\
		&=\frac{1}{(N-d)d}\sum_{kl}\bb E \chi_{ij}^{kl}(\cal A)F(\cal A)+\frac{1}{(N-d)d}\sum_{kl}\bb E \cal A_{ij}(1-\cal A_{ik})\cal A_{kl}\cal A_{jl}F(\cal A)\,,
	\end{aligned}
	\end{equation}
where in the second step we used $1=(1-\cal A_{jl})+\cal A_{jl}$. By Lemma \ref{lem2.1} and $\chi_{ij}^{kl}(\A)\leq \A_{ij}\A_{kl}$,
	\begin{align}
\frac{1}{(N-d)d}\sum_{kl}\bb E \chi_{ij}^{kl}(\cal A)F(\cal A)&=\frac{1}{(N-d)d}\sum_{\substack {kl:
i,j,k,l\\ \emph{distinct}}}\bb E \chi_{ij}^{kl}(\cal A)F(\cal A)+O\Big(\frac{1}{Nd}\Big)\cdot\hspace{-0.2cm}\sum_{\substack {kl:
i,j,k,l\\ \emph{not distinct}}} \bb E |\cal A_{ij}\A_{kl}F(\A)|\nonumber \\ 
&=\frac{1}{(N-d)d}\sum_{\substack {kl:
		i,j,k,l\\ \emph{distinct}}}\bb E \chi_{ik}^{jl}(\cal A)F(\cal A+\xi_{ij}^{kl})+O(N^{-1})\cdot \bb E|F(\A)|\label{2.7}\\
	&=\frac{1}{(N-d)d}\sum_{kl}\bb E \chi_{ik}^{jl}(\cal A)F(\cal A+\xi_{ij}^{kl})+O(N^{-1})\cdot\bb E	\cal M_{ij}(F(\A))\,. \nonumber
\end{align}
Moreover, note that
\begin{equation} \label{2.8}
	\begin{aligned}
&\,\frac{1}{(N-d)d}\sum_{kl}\bb E \chi_{ik}^{jl}(\cal A)F(\cal A)=\frac{1}{(N-d)d}\sum_{kl}\bb E (1-\A_{ij})\A_{ik}(1-\A_{kl})\A_{jl}F(\cal A)\\
	=&\,-\frac{1}{(N-d)d}\sum_{kl}\bb E (1-\A_{ij})(1-\A_{ik})(1-\A_{kl})\A_{jl}F(\cal A)+\bb E (1-\A_{ij})F(\A)\\
	=&\,\frac{1}{(N-d)d}\sum_{kl}\bb E (1-\A_{ij})(1-\A_{ik})\A_{kl}\A_{jl}F(\cal A)\,,
	\end{aligned}
\end{equation}
where in the second and third steps we used $\sum_{kl}(1-\A_{kl})\A_{jl}=\sum_{kl}(1-\A_{ik})\A_{jl}=(N-d)d$. Combining \eqref{2.6} - \eqref{2.8} we get
\begin{equation*}
\begin{aligned}
		\bb E \cal A_{ij}F(\cal A )=&\ \frac{1}{(N-d)d}\sum_{kl}\bb E \chi_{ik}^{jl}(\cal A)\big(F(\cal A+\xi_{ij}^{kl})-F(\A)\big)\\
		&+\frac{1}{(N-d)d}\sum_{kl}\bb E (1-\A_{ij})(1-\A_{ik})\A_{kl}\A_{jl}F(\cal A) \\
		&+\frac{1}{(N-d)d}\sum_{kl}\bb E \cal A_{ij}(1-\cal A_{ik})\cal A_{kl}\cal A_{jl}F(\cal A)+O(N^{-1})\cdot\bb E	\cal M_{ij}(F(\A))\\
		=&\ \frac{1}{(N-d)d}\sum_{kl}\bb E \chi_{ik}^{jl}(\cal A)\big(F(\cal A+\xi_{ij}^{kl})-F(\A)\big)\\
		&+\frac{1}{(N-d)d}\sum_{kl}\bb E (1-\cal A_{ik})\cal A_{kl}\cal A_{jl}F(\cal A)+O(N^{-1})\cdot\bb E	\cal M_{ij}(F(\A))\,.
\end{aligned}
\end{equation*}
Applying $\sum_{kl}(1-\cal A_{ik})\cal A_{kl}\cal A_{jl}=d^2-(\A^3)_{ij}$ to the second term on the RHS, we get the desired result.
\end{proof}

\section{Powers of $\A$: large deviations} \label{sec3}
Let us abbreviate the discrete derivative for any indices $i,j,k,l$ by
\begin{equation} \label{2.9}
	\D_{ij}^{kl}F(\cal A)\deq F(\A +\xi_{ij}^{kl})-F(\A)
\end{equation} 
where $\xi_{ij}^{kl}$ was defined as in \eqref{2.3}. It satisfies the discrete product rule
\begin{equation} \label{3.2}
	\D_{ij}^{kl}(FK)=\D_{ij}^{kl}(F)K+F\D_{ij}^{kl}(K)+\D_{ij}^{kl}(F)\D_{ij}^{kl}(K)\,,
\end{equation}
and 
\begin{equation*}
	\D_{ij}^{kl}(\A)=\xi_{ij}^{kl}\,.
\end{equation*}
We have the following result.

\begin{proposition} \label{prop4.4}
Fix $\tau>0$ and assume $N^{\tau} \leq d \leq N/2$. We have
\begin{equation} \label{1}
	(\A^2)_{ij} -d^2N^{-1} \prec 1+dN^{-1/2}
\end{equation}
uniformly for $i \ne j$, and 
\begin{equation} \label{2}
	(\A^3)_{ij} -d^3N^{-1} \prec d+d^2N^{-1/2}
\end{equation}
uniformly in $i,j$. For fixed integer $r\geq 4$, we have
\begin{equation} \label{3}
	(\A^r)_{ij} -d^rN^{-1} \prec d^{r-2}+Nd^{r-4}
\end{equation}
uniformly in $i,j$.
\end{proposition}
\begin{proof}
(i) Fixed an integer $r \geq 2$. In this step we shall prove
	\begin{equation} \label{3.5}
		(\A^{2r})_{ii}-d^{2r}N^{-1} \prec d^{2r-2}+d^{2r-1}N^{-1/2}
	\end{equation}
uniformly in $i$. By $\sum_{j}(\A^r)_{ij}=d^r$, we have
\begin{equation} \label{3.33}
	0\leq \sum_{j}\big((\A^r)_{ij}-d^rN^{-1}\big)^2=(\A^{2r})_{ii}-d^{2r}N^{-1}\,,
\end{equation}
thus $\cal R_r\deq	(\A^{2r})_{ii}-d^{2r}N^{-1}\geq 0$. Similarly, $\cal R_{r+1}\deq 	(\A^{2r+2})_{ii}-d^{2r+2}N^{-1}\geq 0$. Fix $n \geq 1$. As $A_{ii}=0$, we have
\begin{equation} \label{rockstar}
	\bb E \cal R_r^n=-\frac{d^{2r}}{N}\bb E \cal R_r^{n-1}+\sum_{j} \bb E \A_{ij}(\A^{2r-1})_{ji}\cal R_r^{n-1}=-\frac{d^{2r}}{N}\bb E \cal R_r^{n-1}+\sum_{j:j\ne i} \bb E \A_{ij}(\A^{2r-1})_{ji}\cal R_r^{n-1}\,.
\end{equation}
Applying Lemma \ref{lem2.2} to the last term on RHS of \eqref{rockstar}, with $F(\A)=(\A^{2r-1})_{ji}\cal R_r^{n-1}$, we get
	\begin{align*}
	\bb E \cal R_r^n	=&\,-\frac{d^{2r}}{N}\bb E \cal R_r^{n-1}+\frac{1}{(N-d)d}\sum_{jkl:j \ne i}\bb E \chi_{ik}^{jl}(\cal A)\D_{ij}^{kl}((\A^{2r-1})_{ji}\cal R_r^{n-1})	+\frac{d}{N-d}\sum_{j:j\ne i}\bb E (\A^{2r-1})_{ji}\cal R_r^{n-1}\nonumber\\
		&-\frac{1}{(N-d)d}\sum_{j:j\ne i}\bb E (\A^3)_{ij}(\A^{2r-1})_{ji}\cal R_r^{n-1}+O(N^{-1})\cdot\sum_{j:j\ne i}\bb E\cal M_{ij}((\A^{2r-1})_{ji}\cal R_r^{n-1})\,.
	\end{align*}
As $\max_{ij}(\cal A^{r})_{ij}\leq \max_i \sum_k (\cal A^{r-1})_{ik}=d^{r-1}$ for all $r \geq 2$, we can easily remove the restraint $j\ne i$ in the third and fourth term on RHS of the above, by observing that
\[
	\frac{d}{N-d}\bb E (\cal A^{2r-1})_{ii}\cal R_r^{n-1}=O( d^{2r-2})\cdot\bb E \cal R_r^{n-1}
\quad \mbox{and} \quad 
\frac{1}{(N-d)d}\bb E (\A^3)_{ii}(\A^{2r-1})_{ii}\cal R_r^{n-1}=O( d^{2r-2})\cdot\bb E \cal R^{n-1}_r\,.
\]
As a result,
\begin{align}
	\bb E \cal R_r^n	=&\,-\frac{d^{2r}}{N}\bb E \cal R_r^{n-1}+\frac{1}{(N-d)d}\sum_{jkl:j \ne i}\bb E \chi_{ik}^{jl}(\cal A)\D_{ij}^{kl}((\A^{2r-1})_{ji}\cal R_r^{n-1})	+\frac{d}{N-d}\sum_{j}\bb E (\A^{2r-1})_{ji}\cal R_r^{n-1}\nonumber\\
	&-\frac{1}{(N-d)d}\sum_{j}\bb E (\A^3)_{ij}(\A^{2r-1})_{ji}\cal R_r^{n-1}+O(N^{-1})\cdot\sum_{j:j\ne i}\bb E\cal M_{ij}((\A^{2r-1})_{ji}\cal R_r^{n-1})+O( d^{2r-2})\cdot\bb E \cal R^{n-1}_r\nonumber\\
	\leq &\,-\frac{d^{2r}}{N}\bb E \cal R_r^{n-1}+\frac{1}{(N-d)d}\sum_{jkl:j\ne i}\bb E \chi_{ik}^{jl}(\cal A)\D_{ij}^{kl}((\A^{2r-1})_{ji}\cal R_r^{n-1})	+\frac{d^{2r}}{N-d}\bb E \cal R_r^{n-1}\label{3.7}\\
	&-\frac{d^{2r+1}}{(N-d)N}\bb E \cal R_r^{n-1}+O(N^{-1})\cdot\sum_{j:j\ne i}\bb E\cal M_{ij}((\A^{2r-1})_{ji}\cal R_r^{n-1})+O( d^{2r-2})\cdot\bb E \cal R^{n-1}_r\nonumber\\
	=&\,\frac{1}{(N-d)d}\sum_{jkl:j\ne i}\bb E \chi_{ik}^{jl}(\cal A)\D_{ij}^{kl}((\A^{2r-1})_{ji}\cal R_r^{n-1})+O(N^{-1})\cdot\sum_{j:j\ne i}\bb E\cal M_{ij}((\A^{2r-1})_{ji}\cal R_r^{n-1})\nonumber\\
	&\,+O( d^{2r-2})\cdot\bb E \cal R^{n-1}_r\eqd R_1+R_2+O( d^{2r-2})\cdot\bb E \cal R^{n-1}_r\nonumber\,,
\end{align}
where in the second step we used 
\[
\sum_j (\A^3)_{ij}(\A^{2r-1})_{ji}=(\A^{2r+2})_{ii}=\cal R_{r+1}+d^{2r+2}N^{-1}\geq d^{2r+2}N^{-1}
\]
and $\cal R_r,\cal R_{r+1} \geq 0$, and in the third step we have a cancellation among the three terms involving $\bb E \cal R_{r}^{n-1}$. To estimate the RHS of \eqref{3.7}, note that
\begin{equation} \label{3.11}
(\A^{2r-1})_{ij} \leq d^{2r-2}\,, \quad \max_{jkl}\D_{ij}^{kl} (\A^{2r-1})_{ji} =O(d^{2r-3})\,, \quad \mbox{and} \quad \D_{ij}^{kl} \cal R_r =O( d^{2r-2})\,,
\end{equation}
and together with the product rule \eqref{3.2} and $(N-d)^{-1}\leq 2N^{-1}$,  the term $R_1$ can be bounded (up to a constant factor) by
\begin{equation*} 
\begin{aligned}
& \frac{1}{Nd}\sum_{jkl}\bb E \chi_{ik}^{jl}(\cal A)(\A^{2r-1})_{ji} |\D_{ij}^{kl}\cal R_r^{n-1}|+ \frac{1}{Nd}\sum_{jkl}\bb E \chi_{ik}^{jl}(\cal A)|\D_{ij}^{kl}(\A^{2r-1})_{ji}|( |\D_{ij}^{kl}\cal R_r^{n-1}|+\cal R_r^{n-1})\\
\prec&\,\frac{1}{Nd}\sum_{jkl}\bb E\bigg[ \chi_{ik}^{jl}(\A)(\A^{2r-1})_{ji} \sum_{s=2}^n d^{(2r-2)(s-1)}\cal R_{r}^{n-s}\bigg]+\frac{1}{Nd}\sum_{jkl}\bb E\bigg[ \chi_{ik}^{jl}(\A)d^{2r-3} \sum_{s=1}^n d^{(2r-2)(s-1)}\cal R_{r}^{n-s}\bigg]
\end{aligned}
\end{equation*}
Note that
\[
\sum_{jkl}\chi_{ik}^{jl}(\A)(\A^{2r-1})_{ji} \leq \sum_{jkl}\A_{ik}\A_{jl}(\A^{2r-1})_{ji} = d^2\sum_{j}(\A^{2r-1})_{ji}=d^{2r+1}\ \ \mbox{and} \ \ \sum_{jkl}\chi_{ik}^{jl}(\A)\leq d^2N\,,
\] 
which implies
\begin{equation} \label{3.8}
R_1\prec \sum_{s=1}^n\big(d^{(2r-2)}+d^{2r-1}N^{-1/2}\big)^s\bb E \cal R_r^{n-s}\,.
\end{equation}
More over, by \eqref{3.2} and \eqref{3.11}, it is easy to check that
\begin{equation} \label{et}
\cal M_{ij}((\A^{2r-1})_{ij}\cal R_{r}^{n-1})\leq |(\A^{2r-1})_{ij}\cal R_{r}^{n-1}|+\max_{kl}|\D_{ij}^{kl} ((\A^{2r-1})_{ij}\cal R_{r}^{n-1})| \prec \sum_{s=1}^nd^{(2r-2)s}\bb E \cal R_r^{n-s}\,.
\end{equation}
Hence we have
\begin{equation} \label{3.9}
	R_2 \prec \sum_{s=1}^nd^{(2r-2)s}\bb E \cal R_r^{n-s}\,.
\end{equation}
Combining \eqref{3.7}, \eqref{3.8} and \eqref{3.9}, we get 
\[
\bb E \cal R_r^n \prec \sum_{s=1}^n\big(d^{(2r-2)}+d^{2r-1}N^{-1/2}\big)^s\bb E \cal R_r^{n-s}\leq \sum_{s=1}^n\big(d^{(2r-2)}+d^{2r-1}N^{-1/2}\big)^s(\bb E \cal R_r^{n})^{(n-s)/n}
\]
which implies \eqref{3.5} as desired.

(ii) Fix an integer $r\geq 4$. In this step we shall show that
\begin{equation} \label{3.111}
	(\cal A^{r})_{ij}-d^rN^{-1}\prec d^{r-2}+d^{r-1}N^{-1/2} \,.
\end{equation}
uniformly in $i,j$. More precisely, by $\sum_k (\A^2)_{ik}=d^2$ and $\sum_k (\A^{r-2})_{kj}=d^{s-2}$ we get
\begin{equation*}
	\begin{aligned}
	\Big|	(\cal A^{r})_{ij}-d^rN^{-1}\Big|&=\Big|\sum_k \big((\A^2)_{ik}-d^2N^{-1}\big)\big((\A^{r-2})_{kj}-d^{r-2}N^{-1}\big)\Big|\\
		&\leq \Big(\sum_k \big((\A^2)_{ik}-d^2N^{-1}\big)^2\Big)^{1/2}\Big(\sum_k\big((\A^{r-2})_{kj}-d^{r-2}N^{-1}\big)^2\Big)^{1/2}\\
		&=\big((\A^4)_{ii}-d^4N^{-1}\big)^{1/2}\big((\A^{2r-4})_{jj}-d^{2r-4}N^{-1}\big)^{1/2}\,,
	\end{aligned}
\end{equation*}
and \eqref{3.111} follows from \eqref{3.5}.

(iii) In this step we prove \eqref{1}; the proof of \eqref{2} follows in a similar fashion. Let us denote $\cal S\deq (\A^2)_{ij}-d^2N^{-1}$ for some $i \ne j$. Fix $n \geq 1$. Using Lemma \ref{lem2.2} with $F(\cal A)=\A_{kj}\cal S^{2n-1}$, we have
\begin{equation*}
	\begin{aligned}
\mathbb E \cal S^{2n}&=-d^2N^{-1}\mathbb E \cal S^{2n-1}+\sum_{k} \bb E \A_{ik}\A_{kj} \cal S^{2n-1}\\
&=-d^2N^{-1}\mathbb E \cal S^{2n-1}+\frac{1}{(N-d)d}\sum_{klu:k\ne i}\bb E\chi_{il}^{ku}(\cal A)\D_{ik}^{lu}(\A_{kj}\cal S^{2n-1})+\frac{d}{N-d}\sum_{k:k\ne i}\bb E\A_{kj} \cal S^{2n-1}\\
&\quad\,-\frac{1}{(N-d)d}\sum_{k:k\ne i}\bb E (\A^3)_{ik}\A_{kj}\cal S^{2n-1}+O(N^{-1})\cdot \sum_{k:k\ne i}\cal M_{ik}(\A_{kj} \cal S^{2n-1})\,.
\end{aligned}
\end{equation*}
Similar as in \eqref{3.7} and \eqref{et}, we can remove the restriction $k\ne i$ and estimate the error term in the above, and get
\begin{equation} \label{3.13}
	\begin{aligned}
		\mathbb E \cal S^{2n}
		&=-d^2N^{-1}\mathbb E \cal S^{2n-1}+\frac{1}{(N-d)d}\sum_{klu}\bb E\chi_{il}^{ku}(\cal A)\D_{ik}^{lu}(\A_{kj}\cal S^{2n-1})+\frac{d^2}{N-d}\bb  \cal S^{2n-1}\\
		&\quad\,-\frac{1}{(N-d)d}\bb E (\A^4)_{ij}\cal S^{2n-1}+\sum_{s=1}^{2n}O_{\prec}(1)\cdot\bb E |\cal S^{2n-s}|\,.
	\end{aligned}
\end{equation}
The second term on RHS of \eqref{3.13} can be bounded (up to a constant factor) by
\[
\frac{1}{Nd}\sum_{klu} \bb E \A_{ku}\A_{il} \A_{kj}|\D_{ik}^{ku}\cal S^{2n-1}|+\frac{1}{Nd}\sum_{klu} \bb E |\A_{ku}\A_{il} \D_{ik}^{lu}(\A_{kj})| \big(|\D_{ik}^{ku}\cal S^{2n-1}|+|\cal S^{2n-1}|\big)\,.
\]
Since $i \ne j$, we have $|\D_{ik}^{lu}(\A_{kj})|=O( \delta_{uj}+\delta_{lj}+\delta_{lk}+\delta_{kj}$). Together with the trivial bound $\D_{ik}^{ku}\cal S =O(1)$, we can get the estimate
\begin{equation} \label{3.15}
	\frac{1}{(N-d)d}\sum_{klu}\bb E\chi_{il}^{ku}(\cal A)\D_{ik}^{lu}(\A_{kj}\cal S^{2n-1})\prec\sum_{s=1}^{2n} (1+dN^{-1/2})^s \cdot\bb E |\cal S^{2n-s}|\,.
\end{equation}
Using \eqref{3.111} with $r=4$, we get
\begin{equation} \label{3.17}
	-\frac{1}{(N-d)d}\bb E (\A^4)_{ij}\cal S^{2n-1}=-\frac{d^3}{(N-d)N} \bb E \cal S^{2n-1}+O_\prec(1+dN^{-1/2})\cdot \bb E |\cal S^{2n-1}|\,.
\end{equation}
Combining \eqref{3.13} -- \eqref{3.17} we get
\[
\mathbb E \cal S^{2n}\prec\sum_{s=1}^{2n} (1+dN^{-1/2})^s \cdot\bb E |\cal S^{2n-s}| \leq \sum_{s=1}^{2n} (1+dN^{-1/2})^s \cdot\big(\bb E \cal S^{2n}\big)^{\frac{2n-s}{2n}}\,,
\]
which implies the desired result.

(iv) The proof of \eqref{3} follows from the relation
\begin{equation*}
\begin{aligned}
&\big|(\cal A^{r})_{ij}-d^rN^{-1}\big|\\=&\,\Big|\sum_k \big((\A^2)_{ik}-d^2N^{-1}\big)\big((\A^{r-2})_{kj}-d^{r-2}N^{-1}\big)\Big|\\
\leq&\,  \Big|\sum_{k:k\ne i,j} \big((\A^2)_{ik}-d^2N^{-1}\big)\big((\A^{r-2})_{kj}-d^{r-2}N^{-1}\big)\Big|\\
& +\Big|\big((\A^2)_{ii}-d^2N^{-1}\big)\big((\A^{r-2})_{ij}-d^{r-2}N^{-1}\big)\Big|+\Big|\big((\A^2)_{ij}-d^2N^{-1}\big)\big((\A^{r-2})_{jj}-d^{r-2}N^{-1}\big)\Big|
\end{aligned}
\end{equation*}
and \eqref{1}, \eqref{2}, as well as the trivial bounds $\max_{xy}(\A^2)_{xy}\leq d$, $\max_{xy}(\A^{r-2})_{xy}\leq d^{r-3}$.
\end{proof}

\section{Green function and local semicircle law} \label{sec4}
For the rest of this paper, we shall use the parameter
\begin{equation} \label{3.1}
	q\deq \sqrt{d(N-d)/N}\,.
\end{equation}
Note that $q\asymp \sqrt{d}$ for $d \leq N/2$.
Let us define the projection $P_{\bot}\deq I-\b e \b e^*$, where $\b e=N^{-1/2}(1,...,1)^*$. For $z \in \bb C$ with $\im z>0$, we define the Green function by
\[
G\equiv G(z)\deq P_\bot (q^{-1}\A-z)^{-1} P_\bot\,.
\]
The projection $P_\bot$ was introduced in \cite{BHKY19} to eliminate the large, though trivial impact of $\lambda_1$ in the computations. As a result, the eigenvalues of $G$ are $(q^{-1}\lambda_2-z)^{-1}$,$(q^{-1}\lambda_3-z)^{-1}$,...,$(q^{-1}\lambda_N-z)^{-1}$ and 0. It is easy to check that
\begin{equation} \label{2.1}
	P_\bot \A=\A P_\bot\,,\ G\A=\A G=q(zG+I-\b e \b e^*)\,, \ \mbox{and}\ \sum_i G_{ij}=\sum_j G_{ij}=0\,.
\end{equation}
For $M \in \bb C^{N\times N}$, we denote its normalized trace by $\ul{M}\deq N^{-1}\tr M$. For $x \in \bb R$ and $z \in \bb C$ with $\im z>0$, we denote the semicircle law and its Stieltjes transform by
\[
\varrho(x)=\frac{1}{2\pi}\sqrt{(4-x^2)_+} \quad \mbox{and} \quad m(z)\deq \int \frac{\varrho(x)}{x-z} \,\dd x
\]
respectively. The quantity $m\equiv m(z)$ satisfies $1+zm+m^2=0$. In addition, we have the Ward identity
\begin{equation} \label{ward}
	\sum_{i}|G_{ij}|^2=\sum_{i}G_{ij}G^*_{ji}=(GG^*)_{jj}=\frac{\im G_{jj}}{\eta}\leq \frac{|G_{jj}- m|+\im m}{\eta}\,.
\end{equation}
In the sequel, it is convenient to use the following continuous derivative for any indices $i,j,k,l$,
\begin{equation} \label{4.41}
 \partial_{ij}^{kl}F(\cal A)\deq q\,\partial_t F(\A+t\xi_{ij}^{kl}) \big|_{t=0}	\,,
\end{equation} 
and by Taylor expansion we have
\begin{equation} \label{4.3}
	\D_{ij}^{kl}F(\A)=\sum_{s=1}^\ell \frac{1}{s!q^s}\big(\partial_{ij}^{kl}\big)^{s} F(\A)+\frac{1}{(\ell+1)!q^{\ell+1}}\big(\partial_{ij}^{kl}\big)^{\ell+1}F(\A+\theta \xi_{ij}^{kl})
\end{equation}
for some $\theta \in [0,1]$. We have the differential rule
\begin{equation} \label{diff}
	\partial_{ij}^{kl}G_{xy}=-G_{xi}G_{jy}-G_{xj}G_{iy}-G_{xk}G_{ly}-G_{xl}G_{ky}+G_{xi}G_{ky}+G_{xk}G_{iy}+G_{xj}G_{ly}+G_{xl}G_{jy}\,.
\end{equation}

The following lemma will be useful in our estimates.
\begin{lemma} \label{lemAG}
	Fix $r \in \mathbb N_+$. Suppose $z=O(1)$, then
	\[
	\max_{ij}|(\cal A^rG(z))_{ij}| \prec (d^{r/2}+d^{r-3/2})\Big(\max_{ij}|G(z)_{ij}|+1\Big)\,.
	\] 
\end{lemma}
\begin{proof}
The result follows by repeatedly applying the second relation of \eqref{2.1} $r$ times, and estimating the result using the trivial bound $\max_{ij}(\A^n)_{ij}=O(1+d^{n-1})$ and \eqref{3.1}.
\end{proof}

Fix $\delta>0$, and we define the spectral domain
\begin{equation} \label{spectral}
\b D\equiv 	\b D_{\delta}\deq \{z=E+\ii \eta: |E|\leq \delta^{-1}, N^{-1+\delta}\leq \eta\leq \delta^{-1}\}\,.
\end{equation}
The random graph $\A$ satisfies the following local semicircle law.

\begin{theorem}\label{theorem 4.1}
Assume $N^{\tau} \leq d \leq N/2$ for some fixed $\tau>0$. Fix $\delta\in (0,\tau/10)$. We have
	\[
\max_{ij}\big|	G_{ij}(z)-\delta_{ij}m(z) \big|\prec \frac{1}{(N\eta)^{1/4}}+\frac{1}{d^{1/4}}
	\]
	uniformly for $z \in \b D$.
\end{theorem}

For the rest of this section we prove the next result;
Theorem \ref{theorem 4.1} then follows through a standard stability analysis argument, see e.g.\,\cite[Section 4]{HKR}.

\begin{proposition} \label{prop4.2}
Assume $N^{\tau} \leq d \leq N/2$ for some fixed $\tau>0$. Fix $\delta\in (0,\tau/10)$ and $\nu \in (0,\delta/10)$. Let $z\in \b D$, and suppose that $\max_{ij}|G_{ij}-\delta_{ij}m|\prec \phi $ for some deterministic $\phi \in [N^{-1},N^{\nu}]$ at $z$. Then at $z$ we have
	\[
	\max_{ij}\big|\delta_{ij}+zG_{ij}+\ul{G}G_{ij} \big|\prec (1+\phi)^3\cdot \sqrt{\frac{\phi+\im m}{N\eta}+\frac{1}{d}}\eqd 
	\widetilde{\cal E}\,.
	\]
\end{proposition}

Suppose that
\begin{equation} \label{4.9}
	\max_{ij}\big|\delta_{ij}+zG_{ij}+\ul{G}G_{ij} \big|\prec \Phi
\end{equation}
for some deterministic $\Phi\in [\widetilde{\cal E},N^2]$. Fix $n\in \bb N_+$. Fix indices $i,j$ and denote $P\equiv P_{ij}\deq \delta_{ij}+zG_{ij}+\ul{G}G_{ij}$. Proposition \ref{prop4.2} is an easy consequence of 
\begin{equation} \label{moments}
	\bb E |P|^{2n}\prec \Phi^n\widetilde{\cal E}^{n}\,.
\end{equation}
More precisely, since $n$ is an arbitrary fixed integer, we obtain from Markov's inequality that $P\prec (\Phi \widetilde{\cal E})^{1/2}$. Taking a union bound over indices $i,j$, we get 
\[
\max_{ij}\big|\delta_{ij}+zG_{ij}+\ul{G}G_{ij} \big|\prec  (\Phi \widetilde{\cal E})^{1/2}\,,
\]
provided that \eqref{4.9} holds. Iterating the above, we get Proposition \ref{prop4.2} as desired.

Let us look into the proof of \eqref{moments}. By \eqref{2.1}, we get $P=q^{-1}(\A G)_{ij}+\ul{G}G_{ij}+N^{-1}$, and thus
\begin{equation} \label{4.4}
	\begin{aligned}
		\bb E |P|^{2n}
		&=\frac{1}{q}\sum_{k}\bb E\A_{ik}G_{kj}P^{n-1}\overline{P}^n+\bb E \ul{G}G_{ij}P^{n-1}\overline{P}^{n}+O(N^{-1})\cdot \bb E |P|^{2n-1}\\
		&\eqd \mbox{(I)+(II)}+O(N^{-1})\cdot \bb E |P|^{2n-1}\,.
	\end{aligned}
\end{equation}
We denote $\cal P\deq (\bb E|P|^{2n})^{\frac{1}{2n}}$ and $\cal E\deq (\Phi \widetilde{\cal E})^{1/2}$. It suffices to show that
\begin{equation} \label{4.5}
	\mbox{(I)+(II)}\prec \sum_{r=1}^{2n} \cal E^{r}\cal P^{2n-r}\,.
\end{equation}
To simplify notation, we shall drop the complex conjugates in (I)+(II) (which play no role in the subsequent analysis), and prove
\begin{equation} \label{4.6}
	\mbox{(I)'+(II)'}\deq \frac{1}{q}\sum_{k}\bb E\A_{ik}G_{kj}P^{2n-1}+\bb E \ul{G}G_{ij}P^{2n-1}\prec \sum_{r=1}^{2n} \cal E^{r}\cal P^{2n-r}
\end{equation}
instead of \eqref{4.5}. By triangle inequality and the fact that $|m(z)|=O(1)$, we have
\begin{equation} \label{4.10}
	\max_{ij}|G_{ij}| \prec 1+\phi\,.
\end{equation}
By Lemma \ref{lem2.2}, we have
\begin{equation} \label{4.13}
	\begin{aligned}
		\mbox{(I)'}=&\,\frac{1}{(N-d)dq}\sum_{klx:k\ne i}\bb E \chi_{il}^{kx}(\cal A)\D_{ik}^{lx}(G_{kj}P^{2n-1})	+\frac{d}{(N-d)q}\sum_{k:k\ne i}\bb E G_{kj}P^{2n-1}\\
		&-\frac{1}{(N-d)dq}\sum_{k:k\ne i}\bb E (\A^3)_{ik}G_{kj}P^{2n-1}+O(N^{-1}q^{-1})\cdot\sum_{k:k \ne i}\bb E\cal M_{ik}(G_{kj}P^{2n-1})\\
		=&\,\,\frac{1}{(N-d)dq}\sum_{klx}\bb E \chi_{il}^{kx}(\cal A)\D_{ik}^{lx}(G_{kj}P^{2n-1})	+\frac{d}{(N-d)q}\sum_{k}\bb E G_{kj}P^{2n-1}\\
		&-\frac{1}{(N-d)dq}\sum_{k}\bb E (\A^3)_{ik}G_{kj}P^{2n-1}+O(N^{-1}q^{-1})\cdot\sum_{k:k \ne i}\bb E\cal M_{ik}(G_{kj}P^{2n-1})\\
		&\,-\frac{1}{(N-d)dq}\sum_{lx}\bb E \chi_{il}^{ix}(\cal A)\D_{ii}^{lx}(G_{ij}P^{2n-1})	-\frac{d}{(N-d)q}\bb E G_{ij}P^{2n-1}\\
		&+\frac{1}{(N-d)dq}\bb E (\A^3)_{ii}G_{kj}P^{2n-1}
		\eqd T_1+\cdots+T_7\,.
	\end{aligned}
\end{equation}
By the last relation of \eqref{2.1}, it is easy to see that $T_2=0$. Applying Lemma \ref{lemAG} for $r=3$, we have $(\A^3G)_{ij}\prec (1+\phi)\,d^{3/2}$. Thus
\begin{equation*}
	T_3=O({N^{-1}d^{-3/2}})\cdot \bb E |(\cal A^3G)_{ij}P^{2n-1}| \prec (1+\phi)N^{-1} \bb E |P^{2n-1}|\leq (1+\phi)N^{-1}\cal P^{2n-1}\,.
\end{equation*}
Let us estimate the remainder term $T_4$. By \eqref{3.2}, \eqref{4.3}, \eqref{diff} and \eqref{4.10}, we see that
\[
\cal M_{ik}(G_{kj}P^{2n-1}) \prec |G_{kj}P^{2n-1}|+\max_{xy}|\D_{ik}^{xy}G_{kj}P^{2n-1}| \prec (1+\phi) \sum_{r=1}^{2n} ((1+\phi)^3q^{-1})^{r-1}P^{2n-r}\,,
\]
which implies
\begin{equation} \label{remainder}
	T_4 \prec N^{-1}q^{-1}\cdot N \cdot(1+\phi)\sum_{r=1}^{2n} ((1+\phi)^3q^{-1})^{r-1}\bb EP^{2n-r} \prec \sum_{r=1}^{2n} \cal E^{r}\cal P^{2n-r}\,.
\end{equation}
In the squeal, the remainder term from Lemma \ref{lem2.2} will always be small enough for our purposes, and we shall omit their estimates. Moreover, 
\[
T_5 \prec N^{-1}d^{-3/2}(1+\phi)\bb E\sum_{il} \cal A_{il}\cal A_{ix} \sum_{r=1}^{2n} ((1+\phi)^3q^{-1})^{r-1}P^{2n-r}\prec \sum_{r=1}^{2n} \cal E^{r}\cal P^{2n-r}
\] 
and
\[
T_6,\,T_7 \prec \cal E \cal P^{2n-1}\,.
\]
As a result, \eqref{4.13} simplifies to
\begin{equation} \label{4.166}
	\mbox{(I)'}=T_1+\sum_{r=1}^{2n} O_{\prec}(\cal E^{r})\cdot\cal P^{2n-r}\,.
\end{equation}
To examine the terms in $T_1$, we split according to \eqref{3.2}
\begin{equation} \label{4.188}
	\begin{aligned}
	T_1=&\,\frac{1}{(N-d)dq}\sum_{klx}\bb E \chi_{il}^{kx}(\cal A)\D_{ik}^{lx}(G_{kj})P^{2n-1}+\frac{1}{(N-d)dq}\sum_{klx}\bb E \chi_{il}^{kx}(\cal A)G_{kj}\D_{ik}^{lx}(P^{2n-1})\\
	&+\frac{1}{(N-d)dq}\sum_{klx}\bb E \chi_{il}^{kx}(\cal A)\D_{ik}^{lx}(G_{kj})\D_{ik}^{lx}(P^{2n-1})\eqd T_{1,1}+T_{1,2}+T_{1,3}\,.
	\end{aligned}
\end{equation}

\subsection{Computation of $T_{1,1}$} \label{section 4.1.1} Applying \eqref{4.3} with $\ell=1$, we get
	\begin{align}
	T_{1,1}&=\frac{1}{(N-d)dq^2}\sum_{klx}\bb E \chi_{il}^{kx}(\cal A)\partial_{ik}^{lx}(G_{kj})P^{2n-1}+\frac{1}{2(N-d)dq^3}\sum_{klx}\bb E \chi_{il}^{kx}(\cal A)((\partial_{ik}^{lx})^2G_{kj}(\cal A+\theta \xi_{ik}^{lx}))P^{2n-1}\nonumber\\
	&\eqd T_{1,1,1}+T_{1,1,2} \label{4.16}
	\end{align}
for some $\theta\in [0,1]$. By \eqref{diff}, we get
\begin{multline} \label{4.17}
	T_{1,1,1}=\frac{1}{(N-d)dq^2}\sum_{klx}\bb E \chi_{il}^{kx}(\cal A)P^{2n-1}\big(-G_{kk}G_{ij}-G_{ki}G_{kj}-G_{kl}G_{xj}-G_{kx}G_{lj}\\
	+G_{ki}G_{xj}+G_{kx}G_{ij}+G_{kl}G_{kj}+G_{kk}G_{lj}\big)\eqd \sum_{s=1}^{8}T_{1,1,1,s}\,.
\end{multline}
The term $T_{1,1,1,1}$ is the leading term in our computation. Recall the definition of $\chi$ in \eqref{2.4}. We have 
\begin{equation} \label{4.18}
	\begin{aligned}
	T_{1,1,1,1}=&-\frac{1}{(N-d)dq^2}\sum_k\bb E\big(\A_{ik}(\A^3)_{ik}+d^2-d^2\A_{ik}-(\A^3)_{ik}\big)G_{kk}G_{ij}P^{2n-1}\\
	&=-\frac{1}{(N-d)dq^2}\sum_k\bb E\big(\A_{ik}d^3N^{-1}+d^2-d^2\A_{ik}-d^3N^{-1}\big)G_{kk}G_{ij}P^{2n-1}\\
	&\quad +O_\prec\big((1+\phi)^2 N^{-1/2} \big)\bb E |P|^{2n-1}\\
	&=\frac{d}{Nq^2}\sum_{k}\bb E \A_{ik}G_{kk}G_{ij}P^{2n-1}-\frac{d}{q^2}\bb E \ul{G}G_{ij}P^{2n-1}+O_\prec(\cal E)\bb E |P|^{2n-1}\,.
	\end{aligned}
\end{equation}
Here in the second step we used \eqref{2}, which implies
\begin{equation*}
\begin{aligned}
	&\,\frac{1}{(N-d)dq^2}\sum_{k}\bb E \A_{ik} ((\A^3)_{ik}-d^3N^{-1})G_{kk}G_{ij} P^{2n-1} \\
	\prec&\ \frac{(1+\phi)^2 }{Nd^2} \cdot (d+d^2N^{-1/2}) \cdot N\bb E  |P|^{2n-1} \prec \cal E\,\bb E |P|^{2n-1}
\end{aligned}
\end{equation*}
and similarly 
\[
\frac{1}{(N-d)dq^2}\sum_{k}\bb E  ((\A^3)_{ik}-d^3N^{-1})G_{kk}G_{ij} P^{2n-1} \prec \cal E\,\bb E |P|^{2n-1}\,.
\]
For the first term on RHS of \eqref{4.18}, we again apply Lemma \ref{lem2.2}, this time with $F(\A)=G_{kk}G_{ij}P^{2n-1}$, and get
\begin{multline} \label{4.19}
	\frac{d}{Nq^2}\sum_{k}\bb E \A_{ik}G_{kk}G_{ij}P^{2n-1}=\,\frac{1}{(N-d)Nq^2}\sum_{kxy}\bb E \chi_{ix}^{ky}(\A)\D_{ik}^{xy}(G_{kk}G_{ij}P^{2n-1})\\	+\frac{d^2}{(N-d)q^2}\bb E \ul{G}G_{ij}P^{2n-1}
	-\frac{1}{(N-d)Nq^2}\sum_k\bb E (\A^3)_{ik}G_{kk}G_{ij}P^{2n-1}\\+O\bigg(\frac{d}{N^2q^2}\bigg)\cdot\sum_k\bb E\cal M_{ik}(G_{kk}G_{ij}P^{2n-1})
+\sum_{r=1}^{2n} O_{\prec}(\cal E^{r}) \cal P^{2n-r}
\end{multline}
By \eqref{diff} and \eqref{4.10} one can check that
\begin{equation*}
	\begin{aligned}
\D_{ik}^{xy}(G_{kk}G_{ij}P^{2n-1}) \prec&\,\frac{(1+\phi)^{3}}{q}\sum_{r=1}^{2n}\frac{(1+\phi)^{3r-1}}{q^{r-1}}P^{2n-r}
\\ &+\sum_{r=2}^{2n}\frac{(1+\phi)^{3r-1}}{q^{r-1}}P^{2n-r}\big(|G_{ik}|+|G_{ix}|+|G_{iy}|+|G_{kj}|+|G_{xj}|+|G_{yj}|\big)\,.
	\end{aligned}
\end{equation*}
By \eqref{ward}, we have
\[
\sum_{kxy}\chi_{ix}^{ky}(\A)\big(|G_{ik}|+|G_{ix}|+|G_{iy}|+|G_{kj}|+|G_{xj}|+|G_{yj}|\big) \prec dN^2(1+\varphi) \sqrt{\frac{\phi+\im m}{N\eta}}\,,
\]
and together with $\sum_{kxy}\chi_{ix}^{ky}(\A)\leq d^2N$ and $q\asymp \sqrt{d}$, we get
\begin{equation} \label{4.20}
	\frac{1}{(N-d)Nq^2}\sum_{kxy}\bb E \chi_{ix}^{ky}(\A)\D_{ik}^{xy}(G_{kk}G_{ij}P^{2n-1})\\
	 \prec \sum_{r=1}^{2n} \cal E^{r} \cal P^{2n-r}\,.
\end{equation}
Applying \eqref{2}, we get
\begin{equation} \label{4.21}
	-\frac{1}{(N-d)Nq^2}\sum_k\bb E (\A^3)_{ik}G_{kk}G_{ij}P^{2n-1}=-\frac{d^3}{(N-d)Nq^2}\bb E \ul{G}G_{ij}P^{2n-1}+O_\prec(\cal E)\cdot \cal P^{2n-1}\,.
\end{equation}
In addition, similar to \eqref{remainder}, it can be shown that
\[
O\bigg(\frac{d}{N^2q^2}\bigg)\cdot\sum_k\bb E\cal M_{ij}(G_{kk}G_{ij}P^{2n-1}) \prec \sum_{r=1}^{2n} \cal E^{r} \cal P^{2n-r}\,.
\]
Combing the above and \eqref{4.19} -- \eqref{4.21}, we get
\begin{equation} \label{4.22}
	\frac{d}{Nq^2}\sum_{k}\bb E \A_{ik}G_{kk}G_{ij}P^{2n-1}=\frac{d^2}{Nq^2}\bb E \ul{G}G_{ij}P^{2n-1}+\sum_{r=1}^{2n} O_\prec(\cal E^{r}\cal P^{2n-r})\,.
\end{equation}
Hence \eqref{3.1}, \eqref{4.18} and \eqref{4.22} implies
\begin{equation} \label{4.23}
	T_{1,1,1,1}=-\bb E \ul{G}G_{ij}P^{2n-1}+\sum_{r=1}^{2n} O_\prec(\cal E^{r} \cal P^{2n-r})\,.
\end{equation}

Other terms on RHS of \eqref{4.17} are error terms, and let us estimate them one by one. By $\chi_{il}^{kx}(\A) \leq \A_{il}\A_{kx}$, we have
\[
T_{1,1,1,2} \prec N^{-1}d^{-2}\sum_{klx} \bb E \A_{il}\A_{kx} |P|^{2n-1} |G_{ki}G_{kj}| \prec N^{-1}\sum_{k} \bb E  |P|^{2n-1} |G_{ki}G_{kj}|  \prec \cal E \cal P^{2n-1}\,,
\]
where in the last step we used \eqref{ward} and Jensen's inequality. In addition,
\[
|T_{1,1,1,3}| \prec N^{-1}d^{-2}\sum_{klx} \bb E \A_{il}\A_{kx} |P|^{2n-1} |(1+\phi)G_{xj}| \prec N^{-1}\sum_{x} \bb E  |P|^{2n-1} |(1+\phi)G_{xj}|   \prec \cal E \cal P^{2n-1}\,.
\]
Similarly, we can show that $|T_{1,1,1,5}|+|T_{1,1,1,7}|\prec \cal E \cal P^{2n-1}$. Next, we have
\begin{equation*}
	\begin{aligned}
	T_{1,1,1,4}&=-\frac{1}{(N-d)dq^2}\sum_{klx}\bb E(\A_{il}\A_{kx}+\A_{il}\A_{ik}\A_{kx}\A_{lx}-\A_{il}\A_{kx}\A_{lx}-\A_{il}\A_{ik}\A_{kx})G_{kx}G_{lj}P^{2n-1}\\
	&=-\frac{1}{(N-d)dq^2}\bb E \tr (\A G) (\A G)_{ij} P^{2n-1}+O(N^{-1}d^{-2})\sum_{klx} \bb E (\A_{kx}\A_{lx}+\A_{ik}\A_{kx})(1+\phi)|G_{lj}| |P|^{2n-1}\\
	&\prec d^{-1}\cal P^{2n-1}+N^{-1}\sum_{l}(1+\phi)|G_{lj}| |P|^{2n-1}\prec \cal E\cal P^{2n-1}\,,
	\end{aligned}
\end{equation*}
where in the third step we used Lemma \ref{lemAG}. Similarly, we can show that $|T_{1,1,1,6}|+|T_{1,1,1,8}|\prec \cal E \cal P^{2n-1}$. As a result,
\begin{equation} \label{4.25}
	\sum_{s=2}^{8}T_{1,1,1,s} \prec \cal E \cal P^{2n-1}\,.
\end{equation}
By resolvent identity and \eqref{4.10}, it is easy to see that $((\partial_{ik}^{lx})^2G_{kj}(\cal A+\theta \xi_{ik}^{lx})) \prec (1+\phi)^3$, and thus
\[
T_{1,1,2}\prec \frac{(1+\phi)^3}{(N-d)dq^3}\sum_{klx}\bb E | \chi_{il}^{kx}(\cal A)P^{2n-1}|\leq \frac{(1+\phi)^3}{Nd^{5/2}}\sum_{klx}\bb E | \A_{il}\A_{kx}P^{2n-1}|\prec \frac{(1+\phi)^2}{d^{1/2}} \cal P^{2n-1}\,,
\]
where in the second step we used $q\asymp \sqrt{d}$. Combining the above with \eqref{4.16}, \eqref{4.17}, \eqref{4.23} and \eqref{4.25}, we finish the computation of $T_{1,1}$ by getting
\begin{equation} \label{4.27}
	T_{1,1}=-\bb E \ul{G}G_{ij} P^{2n-1}+\sum_{r=1}^{2n} O_\prec (\cal E^{r}\cal P^{2n-r})\,.
\end{equation}

\subsection{Estimate of $T_{1,2}$} \label{sec4.1.2}
\textit{Case 1.} Let us first illustrate the steps on the dense regime $d \asymp N$. In this case, $q\asymp\sqrt{d}\asymp\sqrt{N}$. Trivially, we have $\D_{ik}^{lx}P \prec \cal E$.  By \eqref{ward}, \eqref{diff} and \eqref{4.10}, we have
\[
q^{-1}\partial_{ik}^{lx} P \prec q^{-1}(1+\phi)^2(|G_{kj}|+|G_{xj}|+|G_{lj}|+|G_{ik}|+|G_{il}|)+q^{-1}\cal E \ \ \mbox{and}\ \ q^{-s}(\partial_{ik}^{lx})^s P \prec q^{-1}\cal E
\] 
for $s \geq 2$. Together with \eqref{3.2} and \eqref{4.3}, it is not hard to see that
\[
\D_{ik}^{lx}(P^{2n-1}) \prec q^{-1}(1+\phi)^2(|G_{kj}|+|G_{xj}|+|G_{lj}|+|G_{ik}|+|G_{il}|)\sum_{r=2}^{2n} \cal E^{r-2}|P|^{2n-r}+q^{-1}\sum_{r=2}^{2n} \cal E^{r-1}|P|^{2n-r}\,.
\]
Together with the trivial bound $\chi_{il}^{kx}(\A)\leq 1$ and \eqref{ward}, we conclude that
\begin{equation*}
	\begin{aligned}
		T_{1,2}&\prec 	N^{-5/2}\sum_{klx}\bb E |G_{kj}\D_{ik}^{lx}(P^{2n-1})|\\
		&\prec N^{-3} (1+\phi)^2 \sum_{klx}\bb E\bigg( |G_{kj}|(|G_{kj}|+|G_{xj}|+|G_{lj}|+|G_{ik}|+|G_{il}|)\sum_{r=2}^{2n} \cal E^{r-2}|P|^{2n-r}\bigg)
		+\sum_{r=2}^{2n} \cal E^{r}\cal P^{2n-r}\\
		&\prec \sum_{r=2}^{2n} \cal E^{r}\cal P^{2n-r}\,.
	\end{aligned}
\end{equation*}

\textit{Case 2.} Now let us examine the general case. The growing complexity is largely due to the fact that we are including the sparse regime $d \ll N$, and as a result we cannot estimate the entries of $\A$ by 1:\ they have to be used in the summations. By \eqref{2.1}, we can rewrite $P$ by $P=q^{-1}(\A G)_{ij}+\ul{G}G_{ij}+N^{-1}$. Using \eqref{3.2}, \eqref{ward}, \eqref{4.3}, \eqref{diff} and \eqref{4.10}, we get
\begin{equation*} 
\begin{aligned}
\D_{ik}^{lx}P=&\,q^{-1}\sum_a  (\D_{ik}^{lx}\A_{ia}) G_{aj}+q^{-1}\sum_a  (\D_{ik}^{lx}\A_{ia})( \D_{ik}^{lx}G_{aj})+q^{-1}\sum_a  \A_{ia}( \D_{ik}^{lx}G_{aj})\\
&+(\D_{ik}^{lx}\ul{G})G_{ij}+(\D_{ik}^{lx} \ul{G})(\D_{ik}^{lx} G_{ij})+ \ul{G}(\D_{ik}^{lx} G_{ij})\\
=&\, q^{-1}\sum_a  (\D_{ik}^{lx}\A_{ia}) G_{aj}+q^{-1}\sum_a  \A_{ia}( \D_{ik}^{lx}G_{aj})+ \ul{G}(\D_{ik}^{lx} G_{ij})+O_\prec((1+\phi)^3)\cdot\bigg( \frac{1}{q^2}+\frac{\im m +\phi}{N\eta q}\bigg)\\
=&\,q^{-1}\sum_a  (\D_{ik}^{lx}\A_{ia}) G_{aj}+q^{-2}\sum_a  \A_{ia}( \partial_{ik}^{lx}G_{aj})+ q^{-1}\ul{G}(\partial_{ik}^{lx} G_{ij})+O_{\prec}(\cal Eq^{-1})\,.
\end{aligned} 
\end{equation*}
Let us denote $P_*\deq \max_{ij}|\delta_{ij}+z\ul{G}_{ij}+\ul{G}G_{ij}|=\max_{ij}|q^{-1}(\A G)_{ij}+\ul{G}G_{ij}+N^{-1}|$. By \eqref{diff} and \eqref{4.10}, it is not hard to see that
\begin{equation*} 
	q^{-2}\sum_a  \A_{ia}( \partial_{ik}^{lx}G_{aj})+ q^{-1}\ul{G}(\partial_{ik}^{lx} G_{ij}) \prec (1+\phi)q^{-1} P_*+N^{-1}\prec (1+\phi)q^{-1} \Phi+N^{-1}\,,
\end{equation*}
where in the last step we also used our assumption \eqref{4.9}. The above shows that heuristically, $\D_{ik}^{lx}$ on $P$ generates some self-similar terms. Hence
\begin{equation} \label{4.29}
	\begin{aligned} 
	\D_{ik}^{lx}P=&\,q^{-1}\sum_a  (\D_{ik}^{lx}\A_{ia}) G_{aj}+O_{\prec}((1+\phi)\Phi q^{-1}+\cal Eq^{-1})\\
	=&\,q^{-1}(G_{kj}+\delta_{ik}G_{ij}+\delta_{il}G_{xj}+\delta_{ix}G_{lj}-G_{lj}-\delta_{il}G_{ij}-\delta_{ik}G_{xj}-\delta_{ix}G_{kj})\\
	&+O_{\prec}((1+\phi)\Phi q^{-1}+\cal Eq^{-1})\,.
	\end{aligned}
\end{equation}
Let us define $X\deq (2n-1)P^{2n-2}+(2n-2)P \D_{ik}^{jl}(P^{2n-3})+(2n-3)P^2 \D_{ik}^{jl}(P^{2n-4})+\cdots+ 2P^{2n-3} \D_{ik}^{jl}(P)$. Note that the trivial estimate $\D_{ik}^{lx}P \prec \cal E$ implies 
\begin{equation} \label{wer}
	X\prec \sum_{r=2}^{2n}\cal E^{r-2} |P|^{2n-r}\,.
\end{equation}
By \eqref{4.29} and \eqref{wer}, we get
\begin{align}
&\D_{ik}^{lx}(P^{2n-1})=(\D_{ik}^{lx}P) X \label{4.30}\\
=&\ q^{-1}(G_{kj}+\delta_{ik}G_{ij}+\delta_{il}G_{xj}+\delta_{ix}G_{lj}-G_{lj}-\delta_{il}G_{ij}-\delta_{ik}G_{xj}-\delta_{ix}G_{kj})X+\sum_{r=2}^{2n} O_{\prec}(\cal E^r)\cdot|P|^{2n-r}\nonumber \,.
\end{align}
By \eqref{3.1}, \eqref{ward} and \eqref{4.10}, it is easy to see that
\begin{equation} \label{4.31}
	\frac{1}{(N-d)dq^2}\sum_{klx}\bb E |\chi_{il}^{kx}(\cal A)G_{kj}q^{-1}(\delta_{ik}G_{ij}+\delta_{il}G_{xj}+\delta_{ix}G_{lj}-\delta_{il}G_{ij}-\delta_{ik}G_{xj}-\delta_{ix}G_{kj})| \prec \cal E^2\,.
\end{equation}
Inserting \eqref{wer}--\eqref{4.31} into \eqref{4.188}, we get
\begin{equation} \label{4.32}
\begin{aligned}
	T_{1,2}=&\ \frac{1}{(N-d)dq}\sum_{klx} \bb E \chi_{il}^{kx}(\A)G_{kj} (\D_{ik}^{lx}P) X\\
	=&\ \frac{1}{(N-d)dq^2}\sum_{klx} \bb E \chi_{il}^{kx}(\A)G_{kj}(G_{kj}-G_{lj})X+\sum_{r=2}^{2n}O_\prec (\cal E^{r})\cdot\cal P^{2n-r}\,.\\
\end{aligned}
\end{equation}
By \eqref{ward} and \eqref{wer}, we have
\begin{equation*}
\begin{aligned}
\frac{1}{(N-d)dq^2}\sum_{klx} \bb E \chi_{il}^{kx}(\A)G^2_{kj}X
\prec \frac{1}{Nd^{2}}\sum_{klx} \bb E\bigg( |\A_{il}\A_{kx}G^2_{kj}| \cdot \sum_{r=2}^{2n} \cal E^{r-2} |P|^{2n-r} \bigg)\prec \sum_{r=2}^{2n} \cal E^{r} \cal P^{2n-r} \,,
\end{aligned}
\end{equation*}
and combing the above with \eqref{4.32} yields
\begin{equation} \label{4.33}
	\begin{aligned}
	T_{1,2}&=-\frac{1}{(N-d)dq^2}\sum_{klx} \bb E \chi_{il}^{kx}(\A)G_{kj}G_{lj}X+\sum_{r=2}^{2n}O_\prec (\cal E^{r})\cdot\cal P^{2n-r}\\
&	\eqd T_{1,2,1}+ \sum_{r=2}^{2n}O_\prec (\cal E^{r})\cdot\cal P^{2n-r}\,.
\end{aligned}
\end{equation}
If we look at the term $T_{1,2,1}$, it contains the factor $\A_{il}G_{lj}$ so we cannot use the smallness of $\sum_l \A_{il}$ and Ward identity at the same time. We (unfortunately) have to apply Lemma \ref{lem2.2} again. Let us abbreviate $F(\A)=(1-\A_{ik})\A_{kx}(1-\A_{lx})G_{kj}G_{lj}X$. Lemma \ref{lem2.2} implies
\begin{align}\label{222}
	T_{1,2,1}=&\,-\frac{1}{(N-d)^2d^2q^2}\sum_{klxyz}\bb E \chi_{iy}^{lz}(\cal A)\D_{il}^{yz}	F(\A)-\frac{1}{(N-d)^2q^2}\sum_{klx}\bb E F(\A)\nonumber\\
	&+\frac{1}{(N-d)^2d^2q^2}\sum_{klx}\bb E (\A^3)_{il}F(\A)+O(N^{-2}d^{-2})\cdot\sum_{klx}\bb E\cal M_{il}(F(\A))+\sum_{r=1}^{2n}O_\prec (\cal E^{r})\cdot\cal P^{2n-r}\nonumber\\
	\eqd&\, T_{1,2,1,1}+\cdots+T_{1,2,1,4}+\sum_{r=1}^{2n}O_\prec (\cal E^{r})\cdot\cal P^{2n-r}\,.
\end{align}
By \eqref{ward} and \eqref{wer} we have
\begin{equation} \label{4.34}
T_{1,2,1,2}\prec N^{-2}d^{-1}\cdot \sum_{klx}\bb E\bigg( |\A_{kx}G_{kj}G_{lj}|\cdot \sum_{r=2}^{2n} \cal E^{r-2} |P|^{2n-r} \bigg)\prec \sum_{r=2}^{2n}\cal E^r \cal P^{2n-r}\,.
\end{equation}
Note that the above estimate works because of the absence of $\A_{il}$. Similarly, we can use \eqref{4.3} and resolvent identity to show that
\begin{equation} \label{4.35}
	T_{1,2,1,1}+T_{1,2,1,4} \prec \sum_{r=2}^{2n}\cal E^r \cal P^{2n-r}\,.
\end{equation}
With the help of \eqref{2}, \eqref{ward} and \eqref{4.10}, we have
\begin{equation} \label{4.37}
	\begin{aligned}
		T_{1,2,1,3}&\prec N^{-2}d^{-3}\sum_{klx} \bb E\bigg( |(\A^3)_{il}\A_{kx}G_{kj}G_{lj}| \cdot \sum_{r=2}^{2n} \cal E^{r-2} |P|^{2n-r} \bigg)\\
		&\prec N^{-1}d^{-2}\sum_{l} \bb E\bigg( |(\A^3)_{il}G_{lj}|\cdot \sum_{r=2}^{2n}\cal E^{r-1}\cal P^{2n-r}\bigg) \\	
		&\prec(1+\phi) N^{-1}d^{-2}\sum_{l} \bb E\bigg( |(\A^3)_{il}-N^{-1}d^3|\cdot \sum_{r=2}^{2n}\cal E^{r-1}\cal P^{2n-r}\bigg) \\
		&\ + N^{-2}d \sum_{l} \bb E\bigg( |G_{lj}|\cdot \sum_{r=2}\cal E^{r-1}\cal P^{2n-r}\bigg)\prec \sum_{r=2}^{2n}\cal E^r \cal P^{2n-r}\,.
	\end{aligned}
\end{equation}
Combining \eqref{4.33} -- \eqref{4.37} we get 
\begin{equation} \label{4.38}
	T_{1,2}\prec \sum_{r=2}\cal E^r \cal P^{2n-r}\,.
\end{equation}

\subsection{Estimate of $T_{1,3}$}
The estimates of $T_{1,3}$ are very similar to those of $T_{1,2}$. In $T_{1,3}$, the factor $G_{kj}$ is replaced by $\D_{ik}^{lx}G_{kj}$, which means we cannot use the Ward identity over summation index $k$. However, we are compensate by the fact that $\D_{ik}^{lx}G_{kj}$ generates at least one factor of $q^{-1}$, which is equivalently good in our estimates. Thus by steps that are very similar to how we estimated $T_{1,2}$, it can be shown that
\begin{equation} \label{4.48}
	T_{1,3} \prec \sum_{r=2}^{2n} \cal E^{r}\cal P^{2n-r}\,.
\end{equation}
Combining \eqref{4.166}, \eqref{4.188}, \eqref{4.27}, \eqref{4.38} and  \eqref{4.48} yields
\[
\mbox{(I)'}=-\bb E \ul{G}G_{ij}P^{2n-1}+\sum_{r=1}^{2n} O_{\prec}(\cal E^{r})\cdot\cal P^{2n-r}\,,
\]
and thus we have \eqref{4.6} as desired. This finishes the proof of Proposition \ref{prop4.2}.

\section{Strong self-consistent equation near the edge} \label{sec5}
To get a more precise description of the spectrum, let us define the shifted Stieltjes transform 
\[
\widehat{m}(z)\deq m\Big(z+\frac{d}{Nq}\Big)\,.
\]
We have
\begin{equation} \label{mhat}
	\widehat{m}^2+\Big(z+\frac{d}{Nq}\Big)\widehat{m}+1=0\,, \quad \mbox{and} \quad \widehat{m}(z)-m(z)=O(N^{-1/4})
\end{equation}
uniformly for $z \in \b D$. Let us write $z=E+\ii \eta$ and $\kappa \deq |(E+\frac{d}{Nq})^2-4|$. It is easy to see that
\begin{equation*} 
	\im \widehat{m}(z) \asymp
	\begin{cases}
		\sqrt{\kappa+\eta} \quad &\mbox{if} \quad |E|\leq 2\\
		\frac{\eta}{\sqrt{\kappa+\eta}} \quad &\mbox{if} \quad |E|>2.
	\end{cases}
\end{equation*}
Having the weak local law at hand, we can relate the entrywise law to the average law in the following sense. A standard consequence of Theorem \ref{theorem 4.1} is the eigenvector delocalization Corollary \ref{cor4.2}, which together with \eqref{ward} implies
\begin{equation} \label{ward2}	
	\sum_{i}|G_{ij}|^2 =\frac{\im G_{jj}}{\eta}\prec \frac{\im \ul{G}}{\eta} \prec \frac{|\ul{G}-\widehat{m}|+\im \widehat{m}}{\eta}\,.
\end{equation}
Comparing \eqref{ward2} with \eqref{ward}, we see that the  improved Ward identity \eqref{ward2} contains the term $|\ul{G}-\widehat{m}|$ instead of $|G_{ii}-m|$, and $|\ul{G}-\widehat{m}|$ is expected to fluctuate on a smaller scale. In addition, by Theorem \ref{theorem 4.1}, triangle inequality and the fact that $m(z)=O(1)$, we have
\begin{equation}  \label{4.42}
	\max_{ij}|G_{ij}| \prec 1
\end{equation}
for all $z \in \b D$. Using \eqref{ward2} and \eqref{4.42}, instead of \eqref{ward} and \eqref{4.10}, we can redo the proof of Proposition \ref{prop4.2} and show that
\[
\max_{i,j}|\delta_{ij}+zG_{ij}+\ul{G}G_{ij}|\prec  \sqrt{\frac{|\ul{G}-\widehat{m}|+\im \widehat{m}}{N\eta}+\frac{1}{d}}\,,
\]
and thus
\[
\max_{i,j}\Big|\delta_{ij}+\Big(z+\frac{d}{Nq}\Big)G_{ij}+\ul{G}G_{ij}\Big|\prec  \sqrt{\frac{|\ul{G}-\widehat{m}|+\im \widehat{m}}{N\eta}+\frac{1}{d}}\,.
\]
The above and \eqref{mhat} imply
\begin{equation} \label{yahaha}
\max_{ij}|G_{ij}-\widehat{m}\delta_{ij}| \prec |\ul{G}-\widehat{m}|+\sqrt{\frac{|\ul{G}-\widehat{m}|+\im \widehat{m}}{N\eta}+\frac{1}{d}}\,.
\end{equation}
In this section we shall prove the following result.

\begin{proposition} \label{prop4.5}
Assume $N^{\tau} \leq d \leq N/2$ for some fixed $\tau>0$. Fix $\delta\in (0,\tau/10)$. Let $z\in \b D$, and suppose that $|\ul{G}-\widehat{m}|\prec \psi$ and $\max_{ij}|G_{ij}-\delta_{ij}\widehat{m}|\prec \psi+\sqrt{\cal E_1}$ for some deterministic $\psi \in [N^{-1},1]$ at $z$, where
\[
\cal E_1\deq \cal E_2+\frac{1}{d}\,, \quad \mbox{and} \quad \cal E_2 \deq \frac{\psi+\im \widehat{m}}{N\eta}\,.
\] 
Then at $z$ we have
	\[
	1+(z+d/(Nq))\ul{G}+\ul{G}^2\prec \cal E_1+\cal E_1^{1/4}\cal E_2^{1/2} (\psi+|z+d/(Nq)+2\widehat{m}|)^{1/2}+d^{-1/2}\psi\eqd 
	\widehat{\cal E}\,.
	\]
\end{proposition}

Fix $n \geq 1$. Let us denote $Q \deq 1+(z+d/(Nq))\ul{G}+\ul{G}^2$. By \eqref{2.1}, $Q=q^{-1}\ul{\A G}+d/(Nq) \cdot\ul{G}+\ul{G}^2+N^{-1}$ we have
\begin{equation*} 
	\begin{aligned}
		\bb E |Q|^{2n}
		&=\frac{1}{Nq}\sum_{ij}\bb E\A_{ij}G_{ji}Q^{n-1}\overline{Q}^n+\bb E (d/(Nq)\cdot\ul{G}+ \ul{G}^2)Q^{n-1}\overline{Q}^{n}+O(N^{-1})\cdot \bb E |Q|^{2n-1}\\
		&\eqd \mbox{(III)+(IV)}+O(N^{-1})\cdot \bb E |Q|^{2n-1}\,.
	\end{aligned}
\end{equation*}
We denote $\cal Q\deq (\bb E|Q|^{2n})^{\frac{1}{2n}}$, it suffices to show that
\begin{equation}  \label{5.5}
	\mbox{(III)+(IV)}\prec \sum_{r=1}^{2n} \widehat{\cal E}^{r}\cal Q^{2n-r}\,.
\end{equation}
To simplify notation, we shall drop the complex conjugates in (III)+(IV) (which play no role in the subsequent analysis), and prove
\begin{equation} \label{4.43}
	\mbox{(III)'+(IV)'}\deq \frac{1}{Nq}\sum_{ij}\bb E\A_{ij}G_{ji}Q^{2n-1}+\bb E (d/(Nq)\cdot\ul{G}+ \ul{G}^2)Q^{2n-1}\prec \sum_{r=1}^{2n} \widehat{\cal E}^{r}\cal Q^{2n-r}
\end{equation}
instead of \eqref{5.5}. By Lemma \ref{lem2.2}, we have
\begin{equation} \label{4.134}
	\begin{aligned}
		\mbox{(III)'}=&\,\frac{1}{(N-d)Ndq}\sum_{ijkl}\bb E \chi_{ik}^{jl}(\cal A)\D_{ij}^{kl}(G_{ij}Q^{2n-1})	+\frac{d}{(N-d)Nq}\sum_{ij}\bb E G_{ij}Q^{2n-1}\\
		&-\frac{1}{(N-d)Ndq}\sum_{ij}\bb E (\A^3)_{ij}G_{ji}Q^{2n-1}+O(N^{-2}q^{-1})\cdot\sum_{ij:i\ne j}\bb E\cal M_{ij}(G_{ji}Q^{2n-1})\\
		&-\frac{1}{(N-d)Ndq}\sum_{ikl}\bb E \chi_{ik}^{il}(\cal A)\D_{ii}^{kl}(G_{ii}Q^{2n-1})-\frac{d}{(N-d)Nq}\sum_{i}\bb E G_{ii}Q^{2n-1}\\
		&+\frac{1}{(N-d)Ndq}\sum_{i}\bb E (\A^3)_{ii}G_{ii}Q^{2n-1}\eqd\, S_1+\cdots+S_7\,.
	\end{aligned}
\end{equation}
By the last relation of \eqref{2.1}, it is easy to see that $S_2=0$. Using Lemma \ref{lemAG} with $r=3$, we have $\sum_{ij}(\A^3)_{ij}G_{ij}=\tr (\A^3G)\prec Nd^{3/2}$. Thus
\begin{equation*}
	S_3=O({N^{-2}d^{-3/2}})\cdot \bb E |\tr (\cal A^3G)Q^{2n-1}| \prec N^{-1}\cal Q^{2n-1}\,.
\end{equation*}
Using resolvent identity and \eqref{4.42}, it can be easily shown that $S_4 \prec \sum_{r=1}^{2n} \widehat{\cal E}^{r}\cal Q^{2n-r}$.
In addition, we have
$
S_5\prec \sum_{r=1}^{2n} \widehat{\cal E}^{r}\cal Q^{2n-r},
$
and
\begin{multline*}
S_7=\frac{1}{(N-d)Ndq}\sum_{ij}\bb E (\A^2)_{ij}\cal A_{ji}G_{ii}Q^{2n-1}\\=\frac{d}{(N-d)N^2q}\sum_{ij}\bb E \cal A_{ji}G_{ii}Q^{2n-1}+O_{\prec}(\widehat{\cal E}) \cal Q^{2n-1}
=\frac{d^2}{(N-d)Nq}\bb E\ul{G}Q^{2n-1}+O_{\prec}(\widehat{\cal E}) \cal Q^{2n-1}\,,
\end{multline*}
where in the second step we used \eqref{1}. Thus $S_6+S_7=-d/(Nq)\bb E\ul{G}Q^{2n-1}$. As a result, \eqref{4.134} simplifies to
\begin{equation} \label{4.1666}
	\mbox{(III)'}=S_1-d/(Nq)\bb E\ul{G}Q^{2n-1}+\sum_{r=1}^{2n} O_{\prec}(\widehat{\cal E}^{r})\cdot\cal Q^{2n-r}\,.
\end{equation}
To examine the terms in $S_1$, we split
\begin{equation} \label{4.1888}
	\begin{aligned}
		S_1=&\,\frac{1}{(N-d)Ndq}\sum_{ijkl}\bb E \chi_{ik}^{jl}(\cal A)\D_{ij}^{kl}(G_{ij})Q^{2n-1}+\frac{1}{(N-d)Ndq}\sum_{ijkl}\bb E \chi_{ik}^{jl}(\cal A)G_{ij}\D_{ij}^{kl}(Q^{2n-1})\\
		&+\frac{1}{(N-d)Ndq}\sum_{ijkl}\bb E \chi_{ik}^{jl}(\cal A)\D_{ij}^{kl}(G_{ij})\D_{ij}^{kl}(Q^{2n-1})\eqd S_{1,1}+S_{1,2}+_{1,3}\,.
	\end{aligned}
\end{equation}

\subsection{Estimates of $S_{1,2}$ and $S_{1,3}$} 
Let us first look at the interaction terms. As we shall see, the steps are much easier compared to those in Section \ref{sec4.1.2}, due to the smallness of $\D_{ij}^{kl}Q$. By \eqref{diff} and \eqref{ward2}, we have
\begin{equation*} 
q^{-1}\partial_{ij}^{kl}Q \prec  q^{-1}|z+d/(Nq)+2\ul{G}|\cdot \frac{\psi+\im \widehat{m}}{N\eta} \prec q^{-1}(|z+d/(Nq)+2\widehat{m}|+\psi) \cal E_2\,,
\end{equation*}
and $q^{-s}(\partial_{ij}^{kl})^sQ  \prec q^{-s}\cal E_2$ for $s \geq 2$. Together with \eqref{3.2}, \eqref{4.3} and $\D_{ij}^{kl}Q \prec \widehat{\cal E}$ we get
\begin{equation} \label{4.49}
	\begin{aligned}
	\D_{ij}^{kl}(Q^{2n-1}) &\prec |\D_{ij}^{kl} Q| \cdot \sum_{r=2}^{2n} \widehat{\cal E}^{r-2} |Q|^{2n-r}\\
	&\prec	(q^{-1}(|z+d/(Nq)+2\widehat{m}|+\psi) \cal E_2+q^{-2}\cal E_2\big)	\sum_{r=2}^{2n} \widehat{\cal E}^{r-2} |Q|^{2n-r}\,.
		\end{aligned}
\end{equation}
By \eqref{4.49} and $\chi_{ik}^{jl}(\A)\leq \A_{ik}\A_{jl}$, we get
	\begin{align}  \label{4.50}
		S_{1,2}&\prec \frac{1}{N^2d^{3/2}}\sum_{ijkl} \bb E \bigg( |\A_{ik}\A_{jl}G_{ij}| (q^{-1}(|z+d/(Nq)+2\widehat{m}|+\psi)\cal E_2 +q^{-2}\cal E_2\big)	\sum_{r=2}^{2n} \widehat{\cal E}^{r-2} |Q|^{2n-r}\bigg)\nonumber\\
		&\prec \frac{(|z+d/(Nq)+2\widehat{m}|+\psi)  +q^{-1}}{N^2d^2}\sum_{i,j,k,l}\bb E\bigg(|\A_{ik}\A_{jl}G_{ij} \cal E_2| \sum_{r=2}^{2n} \widehat{\cal E}^{r-2} |Q|^{2n-r}\bigg)\\
		&\prec \frac{(|z+d/(Nq)+2\widehat{m}|+\psi) +q^{-1}}{N^2}\sum_{i,j}\bb E\bigg( |G_{ij} \cal E_2|\sum_{r=2}^{2n} \widehat{\cal E}^{r-2} |Q|^{2n-r}\bigg) \prec \sum_{r=2}^{2n} \widehat{\cal E}^{r} \cal Q^{2n-r}\nonumber\,.
	\end{align}
Here in the last step we used \eqref{ward2} and Jensen's inequality. Similarly, by \eqref{4.49}, $\chi_{ik}^{jl}(\A)\leq \A_{ik}\A_{jl}$ and $\D_{ij}^{kl}G_{ij}\prec q^{-1}$, we have
	\begin{align} \label{4.51}
\hspace{-0.2cm}		S_{1,3}&\prec \frac{1}{N^2d^{3/2}}\sum_{ijkl} \bb E \bigg( |\A_{ik}\A_{jl}q^{-1} (q^{-1}(|z+d/(Nq)+2\widehat{m}|+\psi)\cal E_2 +q^{-2} \cal E_2\big)	\sum_{r=2}^{2n}\widehat{\cal E}^{r-2} |Q|^{2n-r}\bigg)\nonumber\\
		&\prec \frac{(|z+d/(Nq)+2\widehat{m}|+\psi) +q^{-1}}{N^2d^{5/2}}\sum_{i,j,k,l}\bb E\bigg( |\A_{ik}\A_{jl} \cal E_2|\sum_{r=2}^{2n} \widehat{\cal E}^{r-2} |Q|^{2n-r}\bigg)\prec \sum_{r=2}^{2n} \widehat{\cal E}^{r} \cal Q^{2n-r}\,.
	\end{align}

\subsection{Computation of $S_{1,1}$}
The computation of $S_{1,1}$ is similar to that of $T_{1,1}$ in Section \ref{section 4.1.1}. 
Applying \eqref{4.3} with $\ell=2$, we get
\begin{align*}
	S_{1,1}=&\ \frac{1}{(N-d)Ndq^2}\sum_{ijkl}\bb E \chi_{ik}^{jl}(\cal A)\partial_{ij}^{kl}(G_{ij})Q^{2n-1}+\frac{1}{2(N-d)Ndq^3}\sum_{ijkl}\bb E \chi_{ik}^{jl}(\cal A)((\partial_{ij}^{kl})^2 G_{ij})Q^{2n-1}\nonumber\\
	&+\frac{1}{6(N-d)Ndq^4}\sum_{ijkl}\bb E \chi_{ik}^{jl}(\cal A)((\partial_{ij}^{kl})^3G_{kj}(\cal A+\theta \xi_{ik}^{lx}))Q^{2n-1}
	\eqd S_{1,1,1}+S_{1,1,2}+S_{1,1,3} 
\end{align*}
for some $\theta\in [0,1]$. 

Let us first compute $S_{1,1,1}$. By \eqref{diff}, we get
\begin{multline} \label{4.177}
	S_{1,1,1}=\frac{1}{(N-d)Ndq^2}\sum_{ijkl}\bb E \chi_{ik}^{jl}(\cal A)\big(-G_{ii}G_{jj}-G_{ij}G_{ji}-G_{ik}G_{lj}-G_{il}G_{kj}\\
	+G_{ii}G_{kj}+G_{ik}G_{ij}+G_{ij}G_{lj}+G_{il}G_{jj}\big)Q^{2n-1}\eqd \sum_{s=1}^{8}S_{1,1,1,s}\,.
\end{multline}
Recall the definition of $\chi$ in \eqref{2.4}. We have 
\begin{equation} \label{4.118}
	\begin{aligned}
		S_{1,1,1,1}=&-\frac{1}{(N-d)Ndq^2}\sum_{ij}\bb E\big(\A_{ij}(\A^3)_{ij}+d^2-d^2\A_{ij}-(\A^3)_{ij}\big)G_{ii}G_{jj}Q^{2n-1}\\
		=&-\frac{1}{(N-d)Ndq^2}\sum_{ij}\bb E\Big[\big(\A_{ij}(\A^3)_{ij}+d^2-d^2\A_{ij}-(\A^3)_{ij}\big)\\
&\quad \quad \quad  \cdot\big((G_{ii}-\ul{G})(G_{jj}-\ul{G})+\ul{G}(G_{ii}-\ul{G})+\ul{G}(G_{jj}-\ul{G})+\ul{G}^2\big)Q^{2n-1}\Big]\\
		\eqd&\, S_{1,1,1,1,1}+\cdots+S_{1,1,1,1,4}\,.
	\end{aligned}
\end{equation}
Similar as in \eqref{4.18}, using \eqref{2}, we get
	\begin{align}
		S_{1,1,1,1,1}=&-\frac{1}{(N-d)Ndq^2}\sum_{ij}\bb E\Big[\big(\A_{ij}d^3N^{-1}+d^2-d^2\A_{ij}-d^3N^{-1}\big)(G_{ii}-\ul{G})(G_{jj}-\ul{G})Q^{2n-1}\Big]\nonumber\\
		&\  +O_{\prec}(N^{-1/2}+d^{-1})(\psi^2+\cal E_1)\cal Q^{2n-1}\label{4.59}
		\\
		=&\ \frac{d}{N^2q^2}\sum_{ij}\bb E \A_{ij} (G_{ii}-\ul{G})(G_{jj}-\ul{G})Q^{2n-1}+O_{\prec}(\widehat{\cal E})\cal Q^{2n-1}\nonumber\,. 
	\end{align}
Here in the last step we used $\sum_i(G_{ii}-\ul{G})=0$. Let us denote $\widetilde{F}(\A)\deq (G_{ii}-\ul{G})(G_{jj}-\ul{G})Q^{2n-1}$. Applying Lemma \ref{lem2.2} to the first term on RHS of \eqref{4.59}, we get
	\begin{align} \label{4.60}
		S_{1,1,1,1,1}=&\ \frac{1}{(N-d)N^2q^2}\sum_{ijkl}\bb E \chi_{ik}^{jl}(\cal A)\D_{ij}^{kl} \widetilde{F}(\A)	+\frac{d^2}{(N-d)N^2q^2}\sum_{ij}\bb E \widetilde{F}(\A)\\
		&-\frac{1}{(N-d)N^2q^2}\sum_{ij}\bb E (\A^3)_{ij}\widetilde{F}(\A)+O(dN^{-3}q^{-2})\cdot\sum_{ij}\bb E\cal M_{ij}(F(\A))+\sum_{r=1}^{2n}O_{\prec}(\widehat{\cal E}^r)\cal Q^{2n-r}\nonumber\\
     	=&\ \frac{1}{(N-d)N^2q^2}\sum_{ijkl}\bb E \chi_{ik}^{jl}(\cal A)\D_{ij}^{kl} \widetilde{F}(\A)	
     	-\frac{1}{(N-d)N^2q^2}\sum_{ij}\bb E (\A^3)_{ij}\widetilde{F}(\A)+\sum_{r=1}^{2n}O_{\prec}(\widehat{\cal E}^r)\cal Q^{2n-r}\nonumber\,,
	\end{align}
where in the second step we used $\sum_{ij}\widetilde{F}(\A)=0$. By \eqref{3.2}, \eqref{4.3}, \eqref{diff}, and $\D_{ij}^{kl}Q \prec q^{-1}\widehat{\cal E}$, we have
\[
\D_{ij}^{kl} \widetilde{F}(\A)\prec ((\psi+\sqrt{\cal E_1})q^{-1}+q^{-2})\sum_{r=1}^{2n} \widehat{\cal E}^{r-1}|Q|^{2n-r}+(\psi+\sqrt{\cal E_1})^2q^{-1}\sum_{r=2}^{2n} \widehat{\cal E}^{r-1}|Q|^{2n-r} \prec \sum_{r=1}^{2n}\widehat{\cal E}^{r}|Q|^{2n-r}\,,
\]
which implies
\begin{equation} \label{4.61}
	\frac{1}{(N-d)N^2q^2}\sum_{ijkl}\bb E \chi_{ik}^{jl}(\cal A)\D_{ij}^{kl} \widetilde{F}(\A)\prec \frac{1}{N^3d}\sum_{ijkl} \bb E\bigg( \A_{ik}\A_{jl}\sum_{r=1}^{2n} \widehat{\cal E}^{r}|Q|^{2n-r}\bigg)	\prec \sum_{r=1}^{2n}\widehat{\cal E}^{r}
	\cal Q^{2n-r}\,.
\end{equation}
In addition, \eqref{2} and $\sum_{ij}\widetilde{F}(\A)=0$ imply
\begin{equation} \label{4.62}
	-\frac{1}{(N-d)N^2q^2}\sum_{ij}\bb E (\A^3)_{ij}\widetilde{F}(\A)+O_{\prec}(\widehat{\cal E})\cal Q^{2n-1} \prec \widehat{\cal E}
	\cal Q^{2n-1}\,.
\end{equation}
Combining \eqref{4.60} -- \eqref{4.62} we get
\begin{equation} \label{4.63}
	S_{1,1,1,1,1} \prec \sum_{r=1}^{2n}\widehat{\cal E}^{r}\cal Q^{2n-r}\,.
\end{equation}
Similarly to \eqref{4.59}, we can show that
\[
S_{1,1,1,1,2}=\frac{d}{N^2q^2}\sum_{ij}\bb E \A_{ij} \ul{G}(G_{ii}-\ul{G})Q^{2n-1}+O_{\prec}(\widehat{\cal E})\cal Q^{2n-1}\,.
\]
By first summing over $j$ and then summing over $i$, the first term on RHS of the above vanishes, and thus
\begin{equation} \label{4.67}
S_{1,1,1,1,2}\prec	 \sum_{r=1}^{2n} \widehat{\cal E}^{r}\cal Q^{2n-r}\,.
\end{equation}
Similarly, 
\begin{equation} \label{4.64}
	S_{1,1,1,1,3} \prec \sum_{r=1}^{2n} \widehat{\cal E}^{r}\cal Q^{2n-r}\,.
\end{equation}
Moreover, applying \eqref{3} with $r=4$, we have
\begin{equation} \label{4.65}
	\begin{aligned}
	S_{1,1,1,1,4}&=-\frac{1}{(N-d)Ndq^2}\bb E( \tr \A^4+N^2d^2-Nd^3-Nd^3)\ul{G}^2 Q^{2n-1}\\
	&=-\bb E \ul{G}^2 Q^{2n-1}+\sum_{r=1}^{2n} O_{\prec}(\widehat{\cal E}^{r})\cdot \cal Q^{2n-r}\,.
		\end{aligned}
\end{equation} 
Inserting \eqref{4.63} -- \eqref{4.65} into \eqref{4.118}, we get
\begin{equation} \label{4.66}
	S_{1,1,1,1}=-\bb E \ul{G}^2 Q^{2n-1}+\sum_{r=1}^{2n} O_{\prec}(\widehat{\cal E}^{r})\cdot \cal Q^{2n-r}\,.
\end{equation} 

When $s=2,...,8$, the estimates of $S_{1,1,1,s}$ are relatively simple. By $\chi_{ik}^{jl}\leq \A_{ik}\A_{jl}$ and first summing over indices $k,l$, it is not hard to see that $S_{1,1,1,2}\prec \widehat{\cal E}\cal Q^{2n-1}$. For the next term, we have
\begin{align*}
	S_{1,1,1,3}&=-\frac{1}{(N-d)Ndq^2}\sum_{ijkl}\bb E (\A_{ik}\A_{lj}-\A_{ik}\A_{jl}\A_{lk}-\A_{ik}\A_{jl}\A_{ij}+\A_{ik}\A_{jl}\A_{ij}\A_{lk})G_{ik}G_{lj}Q^{2n-1}\nonumber\\
	&=-\frac{1}{(N-d)Ndq^2}\sum_{ijkl}\bb E \A_{ik}\A_{jl}\A_{ij}\A_{lk}G_{ik}G_{lj}Q^{2n-1}+O(d^{-1})\cal Q^{2n-1}\,, 
\end{align*}
where in the second step we used Lemma \ref{lemAG} with $r=1$. The first term on RHS of the above can be bounded by
\[
O(N^{-2}d^{-2}) \cdot \bb E \Big(\sum_{ijkl}|\A_{ik}\A_{lk}G^2_{jl}|\Big)^{1/2}\Big(\sum_{ijkl}|\A_{jl}\A_{ij}G^2_{ik}|\Big)^{1/2}|Q|^{2n-1}\prec \widehat{\cal E} \cal Q^{2n-1}\,.
\]
Hence $S_{1,1,1,3}\prec \widehat{\cal E}\cal Q^{2n-1}$. Similarly $S_{1,1,1,4}\prec \widehat{\cal E}\cal Q^{2n-1}$. We have
\begin{equation*}
	S_{1,1,1,5}=\frac{1}{(N-d)Ndq^2}\sum_{ijkl}\bb E \chi_{ik}^{jl}(\A) (G_{ii}-\ul{G})G_{kj}Q^{2n-1}+\frac{1}{(N-d)Ndq^2}\sum_{ijkl}\bb E \chi_{ik}^{jl}(\A) \ul{G}G_{kj}Q^{2n-1}\,.
\end{equation*}
By \eqref{ward2}, the first term on RHS of above can be bounded by
\[
O_{\prec}(N^{-2}d^{-2})(\psi+\sqrt{\cal E_1})\sum_{ijkl} \bb E \A_{ik}\A_{jl} |G_{kj}||Q|^{2n-1} \prec N^{-2}(\psi+\sqrt{\cal E_1})\sum_{jk} \bb E |G_{kj}||Q|^{2n-1} \prec \widehat{\cal E}\cal Q^{2n-1}\,;
\]
the second term on RHS can be estimated by
\begin{equation*}
	\begin{aligned}
	&\,\frac{1}{(N-d)Ndq^2}\sum_{kj} \bb E( d^2-2d (\A^2)_{jk}+ (\A^2)_{jk}^2 )\ul{G}G_{kj}Q^{2n-1}\\
	=&\,\frac{1}{(N-d)Ndq^2} \sum_{kj} \bb E (\A^2)_{jk}^2 G_{kj}Q^{2n-1}+O(N^{-1})\cal Q^{2n-1}\\
	=&\,\frac{1}{(N-d)Ndq^2} \sum_{kj} \bb E \big(((\A^2)_{jk}-d^2N^{-1})^2+2d^2N^{-1}((A^2)_{jk}-d^2N^{-1})\big) G_{kj}Q^{2n-1}+O(N^{-1})\cal Q^{2n-1}\\
\prec &\, N^{-2}d^{-2}\bb E \sum_{kj}((\A^2)_{kj}-d^2N^{-1})^2  |Q|^{2n-1}+N^{-3} \bb E \Big(\sum_{kj}((\A^2)_{kj}-d^2N^{-1})^2 \sum_{kj} |G^2_{kj}| \Big)^{1/2} |Q|^{2n-1}\\
&+O(N^{-1})\cal Q^{2n-1} \prec \widehat{\cal E} \cal Q^{2n-1}
\end{aligned}
\end{equation*}
Here in the first step we used $\sum_{k}G_{kj}=0$ and Lemma \ref{lemAG}, in the second step we used $\sum_{k}G_{kj}=0$, and in the last step we used
\[
\sum_{kj}( (\A^2)_{kj}-d^2N^{-1})^2=\tr \A^4 -d^4 \prec d^2 N+N^2
\]
which is a consequence of \eqref{3}. Thus $S_{1,1,1,5}\prec \widehat{\cal E}\cal Q^{2n-1}$, and similarly we have $S_{1,1,1,8}\prec \widehat{\cal E}\cal Q^{2n-1}$. Next, we have
\begin{equation*}
	\begin{aligned}
		S_{1,1,1,6}&=\frac{1}{(N-d)Ndq^2}\sum_{ijk}\bb E (d\A_{ik}-\A_{ik}(\A^2)_{jk}-d\A_{ik}\A_{ij}+(\A^2)_{jk}\A_{ik}\A_{ij})G_{ik}G_{ij}Q^{2n-1}\\
		&=\frac{1}{(N-d)Ndq^2}\bb E \Big(-\sum_{ik} \A_{ik}G_{ik}(\A^2G)_{ik}-d\sum_i\bb E (\A G)_{ii}^2 +\sum_{ijk}(\A^2)_{jk}\A_{ik}\A_{ij}G_{ik}G_{ij}\Big)Q^{2n-1}\\
		&=\frac{1}{(N-d)Ndq^2}\sum_{ijk}\bb E (\A^2)_{jk}\A_{ik}\A_{ij}G_{ik}G_{ij}Q^{2n-1}+O_{\prec}(\widehat{\cal E})\cal Q^{2n-1}\\
		&\prec N^{-2}d^{-2} \bb E \Big[|Q|^{2n-2} \sum_{jk}(\A^2)_{jk}\sum_{i}|G_{ik}G_{ij}|\Big]+\widehat{\cal E}\cal Q^{2n-1} \prec \widehat{\cal E}\cal Q^{2n-1}\,,
	\end{aligned}
\end{equation*}
where in the second step we used $\sum_j G_{ij}=0$, and in the third step we used Lemma \ref{lemAG}. Similarly, we also have $S_{1,1,1,7}\prec \widehat{\cal E}\cal Q^{2n-1}$. 

Now we have finishes estimates of $S_{1,1,1,s}$ for all $s =2,...,8$. Together with \eqref{4.177} and \eqref{4.66} we get
\begin{equation} \label{kkk}
	S_{1,1,1}=-\bb E \ul{G}^2 Q^{2n-1}+\sum_{r=1}^{2n} O_{\prec}(\widehat{\cal E}^{r})\cdot \cal Q^{2n-r}\,.
\end{equation}

The estimate of $S_{1,1,2}$ is very similar to those of $S_{1,1,1,2},...,S_{1,1,1,8}$: by \eqref{diff}, there is at least one off-diagonal factor of $G$ in every term of $S_{1,1,2}$. In addition, compared to $S_{1,1,1,2},...,S_{1,1,1,8}$, there is an extra factor of $q^{-1}\prec \widehat{\cal E}^{1/2}$ in $S_{1,1,2}$. Thus we can show that
\begin{equation} \label{kkkk}
	S_{1,1,2}\prec \widehat{\cal E} \cal Q^{2n-1}\,.
\end{equation}
By resolvent identity, \eqref{diff} and $\max_{ij}|G_{ij}| \prec 1$, it is not hard to see that $(\partial_{ij}^{kl})^3G_{kj}(\cal A+\theta \xi_{ik}^{lx})\prec 1$, hence
\begin{equation} \label{kkkkk}
	S_{1,1,3}\prec N^{-2}d^{-3}\sum_{ijkl} \bb E \A_{ik}\A_{jl}|Q|^{2n-1}\prec \widehat{\cal E}\cal Q^{2n-1}\,. 
\end{equation}
Combining \eqref{kkk} -- \eqref{kkkkk} we have
\begin{equation} \label{k6}
		S_{1,1}=-\bb E \ul{G}^2 Q^{2n-1}+\sum_{r=1}^{2n} O_{\prec}(\widehat{\cal E}^{r})\cdot \cal Q^{2n-r}\,.
\end{equation}
Inserting \eqref{4.1888}, \eqref{4.50}, \eqref{4.51} and \eqref{k6} into \eqref{4.1666}, we get
\[
\mbox{(III)'}=-\bb E \ul{G}^2 Q^{2n-1}-d/(Nq)\bb E\ul{G}Q^{2n-1}+\sum_{r=1}^{2n} O_{\prec}(\widehat{\cal E}^{r})\cdot \cal Q^{2n-r}\,.
\] 
Since $\mbox{(IV)'}\deq \bb E (d/(Nq)\cdot\ul{G}+ \ul{G}^2)Q^{2n-1}$, we have finished the proof of \eqref{4.43}. This concludes the proof of Proposition \ref{prop4.5}.

\section{Edge rigidity and Universality} \label{sec6}
Throughout this section we assume
\begin{equation} \label{2.333}
N^{2/3+\tau}\leq d \leq N/2
\end{equation}
for some fixed $\tau>0$, and fix parameters
\begin{equation} \label{jieluote}
\mu \in (0,\tau/100)\,, \quad \delta \in (0,\mu/10)\,, \quad \nu \in (0,\delta/10)\,.
\end{equation}
We abbreviate
$$
A \deq q^{-1}\A\,.
$$
We shall prove Theorems \ref{thm rigidity} and \ref{theorem main result} at the right edge of the spectrum; the left edge case follows analogously.

\subsection{Improved estimate of averaged Green function}
Recall the notion of $\b D$ in \eqref{spectral}. Let us define the regime 
$$
\b S\equiv \b S_\delta \deq \{z=E+\ii \eta:  2-d/(Nq)+N^{-2/3+\delta}\leq E
\leq \delta^{-1},  N^{-2/3}\leq \eta \leq \delta^{-1}\}\subset \b D\,,
$$
and we use $\kappa\equiv\kappa(E)\deq |(E+d/(Nq))^2-4|$ to denote the distance to edge. We first prove the following consequence of  Theorem \ref{theorem 4.1} and Proposition \ref{prop4.5}.

\begin{proposition} \label{propsition 6.1}
	We have 
	\begin{equation}
		\begin{aligned} \label{6.2}
			|\ul{G}-\widehat{m}| \prec&\  \frac{1}{N(\kappa+\eta)}+\frac{1}{d(\kappa+\eta)^{1/2}}+\frac{1}{N^2(\kappa+\eta)^{5/2}}+\frac{1}{(N\eta)^2(\kappa+\eta)^{1/2}}\\
			&+\frac{1}{N^{2/3}(\kappa+\eta)^{1/2}}
		\end{aligned}
	\end{equation}
for $z \in \b S$, and 
\begin{equation} \label{2.53}
	|\ul{G}-\widehat{m}|\prec \frac{1}{N\eta}+\frac{1}{d^{1/2}}+\frac{(\kappa+\eta)^{1/6}}{(N\eta)^{2/3}}
\end{equation}
for all $z \in \b D$. In addition, we have
\begin{equation} \label{2.54}
	\max_{ij} |G_{ij}-\delta_{ij}\widehat{m}|\prec \frac{1}{(N\eta)^{1/2}}+\frac{1}{d^{1/2}}
\end{equation}
for all $z \in \b D$.
\end{proposition}

\begin{proof}
Since for each fixed $E$, the function $\eta \mapsto\widehat{\cal E}(E+\ii \eta)$ is non-increasing for $\eta>0$, a standard stability analysis (see e.g. \cite[Lemma 5.4]{BEKYY}) and Proposition \ref{prop4.5} imply
	\begin{equation} \label{2.61}
		|\ul{G}-\widehat{m}| \prec \frac{\widehat{\cal E}}{\sqrt{\widehat{\cal E}+\kappa+\eta}}
	\end{equation}
	for all $z \in \b D$. 
	
(i) Let $z \in \b S$. Recall the definition of $\widehat{\cal E}$ in Proposition \ref{prop4.5}. Note that
	\[
\kappa \asymp E+d/(Nq)-2, \quad	\im \widehat{m}\asymp \frac{\eta}{(\kappa+\eta)^{1/2}}\quad \mbox{and} \quad |z+d/(Nq)+2\widehat{m}|\asymp (\kappa+\eta)^{1/2}\,,
	\]
together with Young's inequality we get
		\begin{align}
			\widehat{\cal E} &\prec \cal E_1+\cal E_2^{2/3}(\psi+|z+d/(Nq)+2\widehat{m}|)^{2/3}+d^{-1/2}\psi\nonumber\\
			&\prec \frac{\psi}{N\eta}+\frac{1}{N(\kappa+\eta)^{1/2}}+\frac{1}{d}+\Big(\frac{\psi}{N\eta}+\frac{1}{N(\kappa+\eta)^{1/2}}\Big)^{2/3}(\psi+(\kappa+\eta)^{1/2})^{2/3}+\frac{\psi}{d^{1/2}} \label{1234}\\
			& \prec \frac{\psi}{N\eta}+\frac{1}{N(\kappa+\eta)^{1/2}}+\frac{1}{d}+\frac{\psi^{4/3}}{(N\eta)^{2/3}}+\frac{\psi^{2/3}}{N^{2/3}(\kappa+\eta)^{1/3}}+\frac{\psi^{2/3}(\kappa+\eta)^{1/3}}{(N\eta)^{2/3}} +\frac{1}{N^{2/3}}+\frac{\psi}{d^{1/2}}\nonumber\,.
		\end{align}
	By \eqref{2.61} and the fact that $x \mapsto x/\sqrt{x+\kappa+\eta}$ is increasing, we know that 
		\begin{align}
			|\ul{G}-\widehat{m} |
			\prec &\ \frac{\psi}{N\eta(\kappa+\eta)^{1/2}}+\frac{1}{N(\kappa+\eta)}+\frac{1}{d(\kappa+\eta)^{1/2}}+\frac{\psi}{(N\eta)^{1/2}(\kappa+\eta)^{1/4}}+\frac{\psi^{2/3}}{N^{2/3}(\kappa+\eta)^{5/6}}\nonumber\\
			&+\frac{\psi^{2/3}}{(N\eta)^{2/3}(\kappa+\eta)^{1/6}}+\frac{1}{N^{2/3}(\kappa+\eta)^{1/2}}+\frac{\psi}{d^{1/2}(\kappa+\eta)^{1/2}} \label{kkkkkk}\\
			\prec &\  \frac{1}{N(\kappa+\eta)}+\frac{1}{d(\kappa+\eta)^{1/2}}+\frac{1}{N^2(\kappa+\eta)^{5/2}}+\frac{1}{(N\eta)^2(\kappa+\eta)^{1/2}}+\frac{1}{N^{2/3}(\kappa+\eta)^{1/2}}+N^{-\nu}\psi \nonumber
		\end{align}
	provided that $|\ul{G}-\widehat{m}|\prec \psi$. Here in the first step the fourth term is obtained through
	\[
\frac{\psi^{4/3}}{(N\eta)^{2/3}} \cdot (\widehat{\cal E}+\kappa+\eta)^{-1/2}\leq 	\frac{\psi^{4/3}}{(N\eta)^{2/3}} \cdot \bigg(\frac{\psi^{4/3}}{(N\eta)^{2/3}}\bigg)^{-1/4} \cdot (\kappa+\eta)^{-1/4}=\frac{\psi}{(N\eta)^{1/2}(\kappa+\eta)^{1/4}}\,,
	\] 
	and in last step we used $\kappa+\eta \geq N^{-2/3+\delta}$, $\eta \geq  N^{-2/3}$ and $d \geq N^{2/3+\tau}$. Iterating \eqref{kkkkkk}, we obtain \eqref{6.2}.
	
(ii) Let $z \in \b D$. We have
\begin{equation*} 
	\im \widehat{m} =O(\sqrt{\kappa+\eta})\quad \mbox{and} \quad \quad |z+d/(Nq)+2\widehat{m}| \asymp (\kappa+\eta)^{1/2}\,.
\end{equation*}
Similar to \eqref{1234}, we get
\begin{equation*}
	\begin{aligned}
		\widehat{\cal E} &\prec \frac{\psi}{N\eta}+\frac{(\kappa+\eta)^{1/2}}{N\eta}+\frac{1}{d}+\Big(\frac{\psi}{N\eta}+\frac{(\kappa+\eta)^{1/2}}{N\eta}\Big)^{2/3}(\psi+(\kappa+\eta)^{1/2})^{2/3}+\frac{\psi}{d^{1/2}}\\
		&\prec \frac{\psi}{N\eta}+\frac{(\kappa+\eta)^{1/2}}{N\eta}+\frac{1}{d}+\frac{\psi^{4/3}}{(N\eta)^{2/3}}+\frac{(\kappa+\eta)^{2/3}}{(N\eta)^{2/3}}+\frac{\psi}{d^{1/2}}\,.
	\end{aligned}
\end{equation*}
By \eqref{2.61} and the fact that $x \mapsto x/\sqrt{x+\kappa+\eta}$ is increasing, we get
\begin{equation*} 
	\begin{aligned}
		|\ul{G}-\widehat{m}| &\prec \Big(\frac{\psi}{N\eta}\Big)^{1/2}+\frac{1}{N\eta}+\frac{1}{d^{1/2}}+\frac{\psi^{2/3}}{(N\eta)^{1/3}}+\frac{(\kappa+\eta)^{1/6}}{(N\eta)^{2/3}}+\frac{\psi^{1/2}}{d^{1/4}}
	\end{aligned}
\end{equation*}
provided that $|\ul{G}-\widehat{m}| \prec \psi$. Iterating the above yields \eqref{2.53} as desired.

(iii)  The estimate \eqref{2.54} is a direct consequence of \eqref{yahaha} and \eqref{2.53}.
\end{proof}

\subsection{Proof of Theorem \ref{thm rigidity}}
We shall need the following bound on the magnitude of $\lambda_2,\lambda_N$ as an input, which follows from \cite[Theorem A]{TY19}.

\begin{theorem} \label{thmrefbound}
	 For any fixed $D>0$, there exists a constant $L\equiv L(D)>0$ such that
	\[
	\mathbb P(|\lambda_N/q|\geq L)+\mathbb P(|\lambda_2/q|\geq L) =O_D(N^{-D})\,.
	\]
\end{theorem}

\textit{The upper bound.} Let $z=E+\ii N^{-2/3} \in \b S$. By \eqref{6.2} and $\kappa(E) \geq N^{-2/3+\delta}$, we get
\[
\im \ul{G}(z) \leq |\ul{G}(z)-\widehat{m}(z)|+\im \widehat{m}(z) \prec \frac{1}{N^{1+\nu}\eta}+\frac{\eta}{\sqrt{\kappa+\eta}} \prec \frac{1}{N^{1+\nu}\eta}\,.
\]
This implies that whenever $E\in [2-d/(Nq)+N^{-2/3+\delta},\delta^{-1}]$, with very high probability, there is no eigenvalue of $A$ in the interval $[E-N^{-2/3},E+N^{-2/3}]$. Together with Theorem \ref{thmrefbound}, we get
\begin{equation} \label{upper}
(\lambda_2/q-d/(Nq)-2)_+\prec N^{-2/3+\delta}\,.
\end{equation}

\textit{The lower bound.} Let $\widehat{\b S}\deq \{z=E-d/(Nq)+\ii \eta:2- N^{-2/3+\delta}\leq E\leq  2+N^{-2/3+\delta}, N^{-2/3-\delta/3}
\leq \eta \leq N^{-2/3} \}\subset \b D$, one can easily deduce from \eqref{2.333} and \eqref{2.53} that
\begin{equation} \label{2.533}
	|\ul{G}-\widehat{m}| \prec N^{-1/3+7\delta/18}
\end{equation}
for all $z \in \widehat{\b S}$. Thus
\begin{equation} \label{2.78}
	\im \ul{G}\leq |\ul{G}-\widehat{m}|+\im \widehat{m} \prec N^{-1/3+7\delta/18}+(\eta+\kappa)^{1/2}\prec N^{-1/3+\delta/2}\,.
\end{equation}
Let $f : \bb R \to [0,1]$ be a smooth function such that $f(x)=1$ for $|x+d/(Nq)-2|\leq N^{-2/3+\delta}-N^{-2/3}$, $f(x)=0$ for $|x+d/(Nq)-2|\geq N^{-2/3+\delta}$ and $\|f^{(j)}\|_{\infty}=O(N^{2j/3})$ for all fixed $j\in \bb N_+$. We see that
\begin{align} \label{2.66}
	&\,|\varrho_A([2-d/(Nq)-N^{-2/3+\delta},2-d/(Nq)+ N^{-2/3+\delta}])-N^{-1}\tr f(A)|\nonumber\\
	\nonumber \leq &\,\varrho_A([2-d/(Nq)-N^{-2/3+\delta},2-d/(Nq)- N^{-2/3+\delta}+N^{-2/3}])\nonumber\\
	&+\varrho_A([2-d/(Nq)+N^{2/3+\delta}-N^{-2/3},2-d/(Nq)+ N^{-2/3+\delta}])\\ \nonumber
	\leq &\,2N^{-2/3}\big( \im \ul{G}(2-d/(Nq)-N^{-2/3+\delta}+\ii N^{-2/3})+\im \ul{G}(2-d/(Nq)+N^{-2/3+\delta}+\ii N^{-2/3})\big)  \nonumber\\
	\prec&\, N^{-1+\delta/2}\nonumber\,,
\end{align}
where in the last step we used \eqref{2.78}. Now we compute $N^{-1}\tr f(A)$. Set $l \deq \ceil{3\delta^{-1}}$, and let $\tilde{f}$ be the almost analytic extension of $f$, defined by
\[
\tilde{f}(x)=f(x)+\sum_{j=1}^l\frac{1}{j!}(\ii y)^j f^{(j)}(x)\,.
\]
We define the regime $D \deq \{w=x+\ii y: x \in \bb R, |y|\leq N^{-2/3-\delta/3}\}$. Note that $\lambda_1/q=d/q \notin \supp f$. By \cite[Lemma 3.5]{H}, we have
\begin{equation*} 
	\begin{aligned} 
		&N^{-1}\tr f(A)-\int_{\bb R} f(x)\varrho(x+d/(Nq))\dd x\\
		=&-\frac{\ii}{2\pi}\oint_{\partial D} \tilde{f}(w) (\ul{G}(w)-\widehat{m}(w))\, \dd w+ \frac{1}{\pi}\int_{D} \partial_{\bar{w}}\tilde{f}(w)  (\ul{G}(w)-\widehat{m}(w))\dd^2  w\,.
	\end{aligned}
\end{equation*}
By the trivial bound $|\ul{G}(w)| \leq |y|^{-1}$, we see that
\[
\bigg|\frac{1}{\pi}\int_{D} \partial_{\bar{w}}\tilde{f}(w) (\ul{G}(w)-\widehat{m}(w))\dd^2  w \bigg| = O(1)\cdot \int_D |y^{l-1}f^{(l+1)}(x)|\, \dd^2 w =O\big(N^{-(2/3+\delta/3)l}\cdot N^{2l/3}\big)=O(N^{-1})\,.
\]
By \eqref{2.533} and $\|f\|_1=O(N^{2/3+\delta})$, we have
\[
\bigg|-\frac{\ii}{2\pi}\oint_{\partial D} \tilde{f}(w) (\ul{G}(w)-\widehat{m}(w))\, \dd w\bigg|\prec N^{-1/3+7\delta/18} \oint_{\partial D} |\tilde{f}(w)| \, \dd w\prec N^{-1+25\delta/18}\,.
\]
As a result, we get
\begin{equation} \label{2.77}
	N^{-1}\tr f(A)=\int_{\bb R} f(x)\varrho(x+d/(Nq))\dd x+O_{\prec}(N^{-1+25\delta/18})=\frac{2}{3}N^{-1+3\delta/2}+O_{\prec}(N^{-1+25\delta/18})\,.
\end{equation}
Combining \eqref{2.66} and \eqref{2.77} yields
\[
\varrho_A([2-d/(Nq)-N^{-2/3+\delta},2-d/(Nq)+ N^{-2/3+\delta}])=\frac{2}{3}N^{-1+3\delta/2}+O(N^{-1+25\delta/18})\,
\]
and thus 
$(2-d/(Nq)-\lambda_k/q)_+ \prec  N^{-2/3+\delta}$ for any fixed $k$. Together with \eqref{upper} we finished the proof of Theorem \ref{thm rigidity} on the right side of the spectrum.

\begin{remark}
In Theorem \ref{thm rigidity} we restrict ourselves on the regime $N^{2/3+\tau}\leq d \leq N/2$, where we have the optimal rigidity estimate. It can be deduced from Theorem \ref{theorem 4.1} and Proposition  \ref{prop4.5} that for all $N^{\tau}\leq  d \leq N/2$, we have
\[
\lambda_2 =2\sqrt{d(N-d)/N}(1+o(1))
\]
with very high probability. We do not pursuit it here.
\end{remark}

\subsection{Proof of Theorem \ref{theorem main result}} \label{sec5.3}
Let us define the spectral domain.
\begin{equation*} 
	\widetilde{\b D}\deq \{z=E+\ii \eta: 1\leq E\leq 4, N^{-2/3}\leq \eta\leq 1\}\,.
\end{equation*}
The next result follows from Theorem \ref{thm rigidity} and Proposition \ref{propsition 6.1}.

\begin{corollary} \label{cor1}
For all $z \in \widetilde{\b D}$, we have
\begin{equation*} 
	|\ul{G}-\widehat{m}| \prec \frac{1}{N\eta}+\frac{1}{d^{1/2}}+\frac{(\kappa+\eta)^{1/6}}{(N\eta)^{2/3}}
\end{equation*}
and
\[
\max_{ij} |G_{ij}-\delta_{ij}\widehat{m}|\prec \frac{1}{(N\eta)^{1/2}}+\frac{1}{d^{1/2}}\,.
\]
In addition, we have
	\begin{equation*} 
		|\lambda_2/q+d/(Nq)-2| \prec N^{-2/3}
	\end{equation*}
and
\begin{equation*} 
	\im G \asymp \im \widehat{m}
\end{equation*}
for all $z=E+\ii \eta$ satisfying $1\leq E \leq 4$ and $N^{-2/3+\delta}\leq \eta \leq 1$.	
\end{corollary}

With the help of Corollary \ref{cor1}, one can now obtain Theorem \ref{theorem main result} (at the right spectral edge) using a strategy very similar to that of \cite[Section 9]{BHKY19}. 

More precisely,  by \cite{LY20,AH20} and Corollary \ref{cor1}, one immediately gets that, near the right edge of the spectrum, a Dyson Brownian motion starting at $A$ reaches local equilibrium at time $t_*\gg N^{-1/3}$. Theorem \ref{theorem main result} then follows by comparing the edge statistics of the Dyson Brownian motion at times $0$ and $t_*$.
The main difference in the comparison argument is that one needs to use Lemma \ref{lem2.2} instead of \cite[Corollary 3.2]{BHKY19}. We shall sketch the steps, with emphasis on this difference.

Let us adopt the conventions in \cite{BHKY19}, i.e.\,we consider the constrained GOE $W$ satisfying
\[
\bb EW_{ij}W_{kl}=\frac{1}{N}\Big(\delta_{ik}-\frac{1}{N}\Big)\Big(\delta_{jl}-\frac{1}{N}\Big)+\frac{1}{N}\Big(\delta_{il}-\frac{1}{N}\Big)\Big(\delta_{jk}-\frac{1}{N}\Big)\,.
\]
We have the integration by parts formula
\begin{equation} \label{GOE}
	\bb E W_{ij}F(W)=\frac{1}{N^3}\sum_{k,l}\bb E \big[\partial_tF(W+t\xi_{ij}^{kl})\big|_{t=0}\big]\,.
\end{equation}
The matrix-valued process is defined by
\begin{equation} \label{A}
A(t)\deq \e^{-t/2}A+\sqrt{1-\e^{-t}}W\,,
\end{equation}
and we denote its eigenvalues by $\xi_1(t)\geq \cdots \geq \xi_N(t)$. We define the parameter $s\deq 1-\e^{-t}$. The Green function is defined by $G(t)\equiv G(t;z)\deq P_\bot(A(t)-z)^{-1}P_\bot$. Recall that we use $\varrho(x)$ to denote the semicircle distribution on $[-2,2]$. As $A$ and $W$ have asymptotic eigenvalue densities $\varrho(x+d/(Nq))$ and $\varrho(x)$ respectively, $A(t)$ has asymptotic eigenvalue density 
\[
\varrho(t;x)\deq  \varrho(x+d/(e^{t/2}Nq))\,,
\]
and we define its Stieltjes transform by
\[
m(t;z)\deq m(z+d/(e^{t/2}Nq))\,. 
\]

As in \cite[Sectiom 9.2]{BHKY19}, the next result follows from Corollary \ref{cor1} and \cite{AH20,BHY,LY20}.

\begin{lemma} \label{lemma 6.5}
(i) Let $0 \leq t \ll 1$. We have
	\[
	|\xi_2(t)+d/(e^{t/2}Np)-2| \prec  N^{-2/3}\,.
	\]
	
	(ii) Let $0 \leq t \ll 1$. Uniformly for any $z \in \widetilde{\b D}$, we have
	\[
\big|	\ul{G}(t;z)-m(t;z) \big|\prec \frac{1}{N\eta}+\frac{1}{d^{1/2}}+\frac{(\kappa+\eta)^{1/6}}{(N\eta)^{2/3}}
	\]
	and
	\[
	\max_{ij} |G_{ij}(t;z)-\delta_{ij}m(t;z)|\prec \frac{1}{(N\eta)^{1/2}}+\frac{1}{d^{1/2}}\,.
	\]
	
(iii) Recall the definition of $\mu$ from \eqref{jieluote} and set $t_*=N^{-1/3+\mu}$. Fix $s \in \mathbb R$. We have
\[
\lim_{N \to \infty}\mathbb P_{A(t_*)}\big(N^{2/3}(\xi_2(t_*)+d/(e^{2/t}Nq)-2)\geq s\big)=\lim_{N \to \infty}\mathbb P_{\emph{GOE}}\big(N^{2/3}(\mu_1-2)\geq s\big)\,.
\]
\end{lemma}

The limiting distribution of $\lambda_2$ can be obtained through the following estimate.
\begin{lemma} \label{lemma 6.6}
Let $t_*=N^{-1/3+\mu}$, $\eta=N^{-2/3-\mu}$. For $\kappa\asymp N^{-2/3}$, we define
\[
X_{t}\deq \im \bigg[N \int_{\kappa}^{N^{-2/3+\mu}} \ul {G}(t;2-d/(e^{t/2}Nq)+x+\ii \eta) \dd x \bigg]\,.
\]
Let $L: \bb R \to \bb R$ be a fixed smooth test function with bounded derivatives. We have 
\[
\big|\bb E L(X_{t_*})-\bb E L(X_0)\big| =O( N^{-\tau/4})
\]
\end{lemma}

By Lemma \ref{lemma 6.6} and an analogue of (9.33) in \cite{BHKY19}, we get
\[
\lim_{N \to \infty}\mathbb P_{A}\big(N^{2/3}(\lambda_2/q+d/(Nq)-2)\geq s\big)=\lim_{N \to \infty}\mathbb P_{A(t_*)}\big(N^{2/3}(\xi_2(t_*)-2)\geq s\big)
\]
for any fixed $s \in \mathbb R$. Together with Lemma \ref{lemma 6.5} (iii) we conclude the universality of $\lambda_2$. Analogue results for other non-trivial eigenvalues of $\A$ can be proved in the same way. We omit the details.

\begin{proof}[Proof of Lemma \ref{lemma 6.6}]
Let us abbreviate $G\equiv G(t)$. We have
\begin{equation} \label{6.20}
	\begin{aligned}
		\frac{\dd }{\dd t} \mathbb E L(X_t)=\bb E \bigg[L'(X_t)\im \int_\kappa^{N^{-2/3}+\mu}-\sum_{ij} \dot A_{ij}(t)(G^2)_{ij}+Nd/(2e^{t/2}Nq)\ul{G^2}\, \dd x\bigg]\,.
	\end{aligned}
\end{equation}
By \eqref{4.41}, \eqref{GOE}, and \eqref{A}
\begin{equation} \label{6.21}
	\begin{aligned}
-\sum_{ij} \bb E\dot A_{ij}(t)L'(X_t)(G^2)_{ij}&=\frac{1}{2}\sum_{ij}\bb E \bigg[\bigg(e^{-t/2}A_{ij}-\frac{\e^{-t}}{\sqrt{1-e^{-t}}}W_{ij}\bigg)L'(X_t)(G^2)_{ij}\bigg]\\
&=\frac{\e^{-t/2}}{2q}\sum_{ij} \bb E \A_{ij} L'(X_t) (G^2)_{ij}-\frac{\e^{-t/2}}{2N^3}\sum_{ijkl}\bb E \partial_{ij}^{kl}(L'(X_t)(G^2)_{ij})\,.
	\end{aligned}
\end{equation}
By Lemma \ref{lem2.2}, the first term on RHS of \eqref{6.21} can be computed by
\begin{equation*}
	\begin{aligned}
		&\,\frac{
		\e^{-t/2}}{2(N-d)dq}\sum_{ijkl}\bb E \chi_{ik}^{jl}(\A)\D_{ij}^{kl}(L'(X_t)(G^2)_{ij})	+\frac{\e^{-t/2}d}{2(N-d)q}\sum_{ij}\bb E L'(X_t)(G^2)_{ij}\\
		&-\frac{e^{-t/2}}{2(N-d)dq}\sum_{ij}\bb E (\A^3)_{ij}L'(X_t)(G^2)_{ij}+O(N^{-1}q^{-1})\cdot \sum_{ij}\bb E\cal M_{ij}(L'(X_t)(G^2)_{ij})\\
		&-\frac{\e^{-t/2}}{2(N-d)dq}\sum_{ikl}\bb E \chi_{ik}^{il}(\A)\D_{ii}^{kl}(L'(X_t)(G^2)_{ii})	-\frac{\e^{-t/2}d}{2(N-d)q}\sum_{i}\bb E L'(X_t)(G^2)_{ii}\\
		&+\frac{e^{-t/2}}{2(N-d)dq}\sum_{i}\bb E (\A^3)_{ii}L'(X_t)(G^2)_{ii}	
		\eqd Y_1+\cdots+Y_7
		\end{aligned}
\end{equation*}
By $\sum_{i} G_{ij}=0$, we have $Y_2=0$. By Lemma \ref{lemma 6.5} (ii), one can deduce that
\[
\im \ul{G}(2+x+\ii N^{-2/3}) \prec N^{-1/3+\mu} \quad \mbox{and} \quad \max_{ij}|G_{ij}(2+x+\ii N^{-2/3})| \prec 1
\]
for $\kappa \leq x\leq N^{-2/3+\mu}$. Since  $y \im[\ul{G}(2+x+\ii y)]$ is a monotone decreasing function of $y$, we get
\begin{equation} \label{6.22}
\im \ul{G}(2+x+\ii 
\eta) \prec N^{-1/3+2\mu} \quad \mbox{and} \quad \max_{ij}|G_{ij}(2+x+\ii\eta)| \prec  N^{\mu}
\end{equation} 
for $\kappa \leq x\leq N^{-2/3+\mu}$. From the above and \eqref{ward} we can deduce that 
\[
Y_4 \prec Nq^{-1} \frac{N^{-1/3+2\mu}}{\eta}=N^{4/3+3\mu}q^{-1}\,.
\]
Similar as in Lemma \ref{lemAG}, we can apply the second relation of \eqref{2.1} and show that
\[
(\A^3G^2)_{ii}\prec d^{3/2} \frac{\im \ul{G}}{\eta} \prec d^{3/2} N^{1/3+3\mu}\,,
\]
where in the last step we used \eqref{6.22}. This implies $Y_3 \prec N^{1/3+3\mu}$. Similar to the estimates of $S_5,S_6,S_7$ in \eqref{4.134}, we can show that $Y_5\prec N^{1/3+3\mu}$ and
\[
Y_6+Y_7=-\frac{\e^{-t/2}Nd}{2(N-d)q}\bb E\ul{G^2}+O_{\prec}(N^{5/6+10\mu}d^{-1/2})\,.
\] 
Next, by \eqref{4.3}, we get
\begin{equation} \label{6.23}
	Y_1=\frac{\e^{-t/2}}{{2(N-d)dq^2}}\sum_{ijkl}\bb E \chi_{ik}^{jl}(\A)\partial_{ij}^{kl}(L'(X_t)(G^2)_{ij})+O_\prec(N^{4/3+10\mu}d^{-1/2})\,.
\end{equation}
Let us denote the first term on RHS of the above by $Y_{1,1}$. Using Lemma \ref{lem2.2} with $F(\A)=(1-\A_{ij})\A_{jl}(1-\A_{kl})\partial_{ij}^{kl}(L'(X_t)(G^2)_{ij})$, we get
\begin{equation*}
\begin{aligned}
	Y_{1,1}=&\,\frac{\e^{-t/2}}{2(N-d)^2d^2q^2}\sum_{ijklab}\bb E \chi_{ia}^{kb}(\A)\D_{ik}^{ab}F(\A)	+\frac{\e^{-t/2}}{2(N-d)^2q^2}\sum_{ijkl}\bb E F(\A)\\
	&-\frac{\e^{-t/2}}{2(N-d)^2d^2q^2}
	\sum_{ijkl}\bb E (\A^3)_{ik}F(\A)+O(N^{-2}d^{-2})\cdot\sum_{ijkl}\bb E\cal M_{ik}(F(\A))\\
	=&\,\frac{\e^{-t/2}}{2(N-d)^2q^2}\sum_{ijkl}\bb E F(\A)-\frac{\e^{-t/2}}{2(N-d)^2d^2q^2}
	\sum_{ijkl}\bb E (\A^3)_{ik}F(\A)+O_{\prec}(N^{1-\tau/3})\,,
\end{aligned}
\end{equation*}
where in the second step we used \eqref{2}, \eqref{4.3} and \eqref{6.22}. By Proposition \ref{prop4.4},  the second term on RHS of the above can be estimated by
\begin{equation*}
	\begin{aligned}
		&-\frac{\e^{-t/2}d}{2N(N-d)^2q^2} \sum_{ijkl} \bb EF(\A)+O_\prec(N^{-2}d^{-3})\sum_{ijkl} \bb E |(\A^3)_{ik}-d^3N^{-1}| |F(\A)|\\
		=&-\frac{\e^{-t/2}d}{2N(N-d)^2q^2} \sum_{ijkl} \bb EF(\A)	+O_\prec(N^{-2}d^{-3}\cdot N^{4/3+10\mu}d)\sum_{ik} \bb E |(\A^3)_{ik}-d^3N^{-1}| \\
		=&	-\frac{\e^{-t/2}d}{2N(N-d)^2q^2} \sum_{ijkl} \bb EF(\A)+O_{\prec}(N^{5/6+10\mu})\,.
	\end{aligned}
\end{equation*}
Since $N^{2/3+\tau}\leq d\leq N/2$, we have
\begin{equation} \label{6.24}
	\begin{aligned}
Y_{1,1}&=\frac{\e^{-t/2}}{2(N-d)^2q^2}\sum_{ijkl}\bb E F(\A)	-\frac{e^{-t/2}d}{2N(N-d)^2q^2} \sum_{ijkl} \bb EF(\A)+O_{\prec}(N^{1-\tau/3})\\
&=\frac{\e^{-t/2}}{2(N-d)Nq^2}\sum_{ijkl} (1-\A_{ij})\A_{jl}(1-\A_{kl})\partial_{ij}^{kl}(L'(X_t)(G^2)_{ij})+O_{\prec}(N^{1-\tau/3})\,.
	\end{aligned}
\end{equation}
Comparing to \eqref{6.23}, we see that heuristically, the above replaces the factor $\A_{ik}$ in $Y_{1,1}$ by $dN^{-1}$, with a small error. Repeating \eqref{6.24} three times we get
\[
Y_{1,1}=\frac{\e^{-t/2}d(N-d)}{2q^2N^4}\sum_{ijkl}\partial_{ij}^{kl}(L'(X_t)(G^2)_{ij})+O_{\prec}(N^{1-\tau/3})=\frac{\e^{-t/2}}{2N^3}\sum_{ijkl}\partial_{ij}^{kl}(L'(X_t)(G^2)_{ij})+O_{\prec}(N^{1-\tau/3})\,,
\]
and together with \eqref{6.23} yields
\[
Y_1=\frac{\e^{-t/2}}{2N^3}\sum_{ijkl}\partial_{ij}^{kl}(L'(X_t)(G^2)_{ij})+O_{\prec}(N^{1-\tau/3})\,.
\]
Inserting the above results of $Y_1,\dots,Y_7$ to \eqref{6.21}, we get
\[
-\sum_{ij} \bb E\dot A_{ij}(t)L'(X_t)(G^2)_{ij}=-\frac{\e^{-t/2}Nd}{2(N-d)q}\bb E\ul{G^2} +O_{\prec}( N^{1-\tau/3})\,.
\]
Together with \eqref{6.20} we conclude the proof.
\end{proof}

	{\small
	
	\bibliography{bibliography} 
	
	\bibliographystyle{amsplain}
}

\vspace{1cm}
\noindent
Yukun He\\
Department of Mathematics\\
City University of Hong Kong\\
 Email: \href{mailto:yukunhe@cityu.edu.hk}{yukunhe@cityu.edu.hk}
\end{document}